\documentclass[fmp,times]{cuparticle}

\usepackage{latexsym,amssymb}
\usepackage{amsmath}

\newtheorem{theorem}{Theorem}
\newtheorem{proposition}[theorem]{Proposition}

\newtheorem{lemma}[theorem]{Lemma}
\newtheorem{corollary}[theorem]{Corollary}

\newdefinition{definition}{Definition}
\newtheorem{remark}[theorem]{Remark}

\newproof{proof}{Proof}

\volume{???}
\doi{???}

\newcommand\R{\mathbb{R}}
\newcommand\Z{\mathbb{Z}}

\newcommand\C{\mathbb{C}}

\renewcommand\Re{{\operatorname{Re\,}}}
\renewcommand\Im{{\operatorname{Im\,}}}
\newcommand\eps{\varepsilon}

\begin{document}

\authorheadline{Brad Rodgers and Terence Tao}
\runningtitle{The De Bruijn-Newman constant is non-negative}
 
\begin{frontmatter}

\title{THE DE BRUIJN-NEWMAN CONSTANT IS NON-NEGATIVE}
\author[1]{Brad Rodgers}
 
\address[1]{Dept. of Math. and Stat., Queen's University, Kingston ON K7L 3N6, Canada \ead{brad.rodgers@queensu.ca}}

\author[2]{Terence Tao}
 
\address[2]{Department of Mathematics, UCLA, Los Angeles CA 90095, USA   \ead{tao@math.ucla.edu}}

\received{???}
 
\begin{abstract}

For each $t \in \R$, define the entire function
$$ H_t(z) \coloneqq \int_0^\infty e^{tu^2} \Phi(u) \cos(zu)\ du$$
where $\Phi$ is the super-exponentially decaying function
$$ \Phi(u) \coloneqq \sum_{n=1}^\infty (2\pi^2  n^4 e^{9u} - 3\pi n^2 e^{5u} ) \exp(-\pi n^2 e^{4u} ).$$
Newman showed that there exists a finite constant $\Lambda$ (the \emph{de Bruijn-Newman constant}) such that the zeroes of $H_t$ are all real precisely when $t \geq \Lambda$.  The Riemann hypothesis is the equivalent to the assertion $\Lambda \leq 0$, and Newman conjectured the complementary bound $\Lambda \geq 0$.

In this paper we establish Newman's conjecture.  The argument proceeds by assuming for contradiction that $\Lambda < 0$, and then analyzing the dynamics of zeroes of $H_t$ (building on the work of Csordas, Smith, and Varga) to obtain increasingly strong control on the zeroes of $H_t$ in the range $\Lambda < t \leq 0$, until one establishes that the zeroes of $H_0$ are in local equilibrium, in the sense that locally behave (on average) as if they were equally spaced in an arithmetic progression, with gaps staying close to the global average gap size.  But this latter claim is inconsistent with the known results about the local distribution of zeroes of the Riemann zeta function, such as the pair correlation estimates of Montgomery.
\end{abstract}
 
\end{frontmatter}  

\section{Introduction}

Let $H_0 \colon \C \to \C$ denote the function
\begin{equation}\label{hoz}
 H_0(z) \coloneqq \frac{1}{8} \xi\left(\frac{1}{2} + \frac{iz}{2}\right),
\end{equation}
where $\xi$ denotes the Riemann xi function
\begin{equation}\label{sas}
 \xi(s) \coloneqq \frac{s(s-1)}{2} \pi^{-s/2} \Gamma\left(\frac{s}{2}\right) \zeta(s)
\end{equation}
and $\zeta$ is the Riemann zeta function.
Then $H_0$ is an entire even function with functional equation $H_0(\overline{z}) = \overline{H_0(z)}$, and the Riemann hypothesis is equivalent to the assertion that all the zeroes of $H_0$ are real.

It is a classical fact (see \cite[p. 255]{titch}) that $H_0$ has the Fourier representation
$$ H_0(z) = \int_0^\infty \Phi(u) \cos(zu)\ du$$
where $\Phi$ is the super-exponentially decaying function
\begin{equation}\label{phidef}
 \Phi(u) \coloneqq \sum_{n=1}^\infty (2\pi^2  n^4 e^{9u} - 3\pi n^2 e^{5u} ) \exp(-\pi n^2 e^{4u} ).
\end{equation}
The sum defining $\Phi(u)$ converges absolutely for negative $u$ also.  From Poisson summation one can verify that $\Phi$ satisfies the functional equation $\Phi(u) = \Phi(-u)$ (i.e., $\Phi$ is even).

De Bruijn \cite{debr} introduced the more general family of functions $H_t \colon \C \to \C$ for $t \in \R$ by the formula
\begin{equation}\label{htdef}
 H_t(z) \coloneqq \int_0^\infty e^{tu^2} \Phi(u) \cos(zu)\ du.
\end{equation}
As noted in \cite[p.114]{csv}, one can view $H_t$ as the evolution of $H_0$ under the backwards heat equation $\partial_t H_t(z)= -\partial_{zz} H_t(z)$.
As with $H_0$, each of the $H_t$ are entire even functions with functional equation $H_t(\overline{z}) = \overline{H_t(z)}$.  From results of P\'olya \cite{polya} it is known that $H_t$ has purely real zeroes for some $t$ then $H_{t'}$ has purely real zeroes for all $t'>t$.  
De Bruijn showed that the zeroes of $H_t$ are purely real for $t \geq 1/2$.  Strengthening these results, Newman \cite{newman} showed that there is an absolute constant $-\infty < \Lambda \leq 1/2$, now known as the \emph{De Bruijn-Newman constant}, with the property that $H_t$ has purely real zeroes if and only if $t \geq \Lambda$.  The Riemann hypothesis is then clearly equivalent to the upper bound $\Lambda \leq 0$.  Newman conjectured the complementary lower bound $\Lambda \geq 0$, and noted that this conjecture asserts that if the Riemann hypothesis is true, it is only ``barely so''.  As progress towards this conjecture,  several lower bounds on $\Lambda$ were established: see Table \ref{lower}.

\begin{table}[ht]
\caption{Previous lower bounds on $\Lambda$.  Dates listed are publication dates.  The final four results use the method of Csordas, Smith, and Varga \cite{csv}.}\label{lower}
\begin{tabular}{|l|l|}
\hline 
Lower bound on $\Lambda$ & Reference \\
\hline 
$-\infty$ & Newman 1976 \cite{newman} \\
$-50$ & Csordas-Norfolk-Varga 1988 \cite{cnv} \\
$-5$ & te Riele 1991 \cite{tr} \\
$-0.385$ & Norfolk-Ruttan-Varga 1992 \cite{nrv} \\
$-0.0991$ & Csordas-Ruttan-Varga 1991 \cite{crv} \\
\hline
$-4.379 \times 10^{-6}$ & Csordas-Smith-Varga 1994 \cite{csv} \\
$-5.895 \times 10^{-9}$ & Csordas-Odlyzko-Smith-Varga 1993 \cite{cosv} \\
$-2.63 \times 10^{-9}$ & Odlyzko 2000 \cite{odlyzko} \\
$-1.15 \times 10^{-11}$ & Saouter-Gourdon-Demichel 2011 \cite{saouter} \\
\hline
\end{tabular}
\end{table}

We also mention that the upper bound $\Lambda \leq 1/2$ of de Bruijn \cite{debr} was sharpened slightly\footnote{Added in press: this bound has recently been improved to $\Lambda \leq 0.22$ in \cite{polymath}.} by Ki, Kim, and Lee \cite{kkl} to $\Lambda < 1/2$.  See also \cite{stopple}, \cite{cmmrs} on work on variants of Newman's conjecture, and \cite[Chapter 5]{broughan} for a survey.

The main result of this paper is to affirmatively settle Newman's conjecture:

\begin{theorem}\label{main}  One has $\Lambda \geq 0$.
\end{theorem}

We now discuss the methods of proof.  Starting from the work of Csordas-Smith-Varga \cite{csv}, the best lower bounds on $\Lambda$ were obtained by exploiting the following repulsion phenomenon: if $\Lambda$ was significantly less than zero, then adjacent zeroes of $H_0$ (or of the Riemann $\xi$ function) cannot be too close to each other (as compared with the other nearby zeroes). See \cite[Theorem 1]{csv} for a precise statement. 
  In particular, a negative value of $\Lambda$ gives limitations on the quality of ``Lehmer pairs'' \cite{lehmer}, which roughly speaking refer to pairs of adjacent zeroes of the Riemann zeta function that are significantly closer to each other than the average spacing of zeroes at that level.  The lower bounds on $\Lambda$ in \cite{csv}, \cite{cosv}, \cite{odlyzko}, \cite{saouter} then follow from numerically locating Lehmer pairs of increasingly high quality. (See also \cite{stopple-2} for a refinement of the Lehmer pair concept used in the above papers.)

In principle, one could settle Newman's conjecture by producing an infinite sequence of Lehmer pairs of arbitrarily high quality.  As suggested in \cite{odlyzko}, we were able to achieve this under the Gaussian Unitary Ensemble (GUE) hypothesis on the asymptotic distribution of zeroes of the Riemann zeta function; we do not detail this computation here\footnote{A sketch of the argument may be found at {\tt terrytao.wordpress.com/2018/01/20}.} as it is superseded by our main result.  However, without the GUE hypothesis, the known upper bounds on narrow gaps between zeroes (e.g. \cite{cgg}) do not appear to be sufficient to make this strategy work, even if one assumes the Riemann Hypothesis (which one can do for Theorem \ref{main} without loss of generality). Instead, we return to the analysis in \cite{csv} and strengthen the repulsion phenomenon to a \emph{relaxation to local equilibrium} phenomenon: if $\Lambda$ is negative, then the zeroes of $H_0$ are not only repelled from each other, but will nearly always be arranged locally as an approximate arithmetic\footnote{To illustrate the equilibrium nature of arithmetic progressions under backwards heat flow, consider the entire functions $F_t(z) \coloneqq e^{tu^2} \cos(zu)$ for some fixed real $u>0$.  These functions all have zeroes on the arithmetic progression $\{ \frac{2\pi (k+\tfrac{1}{2})}{u}: k \in \Z \}$ and solve the backwards heat equation $\partial_t F_t = - \partial_{zz} F$.} progression, with the gaps between zero mostly staying very close to the global average gap that is given by the Riemann-von Mangoldt formula.  

To obtain the local relaxation to equilibrium under the hypothesis that $\Lambda < 0$ requires a sequence of steps in which we obtain increasingly strong control on the distribution of zeroes of $H_t$ for $\Lambda < t \leq 0$ (actually for technical reasons we will need to move $t$ away from $\Lambda$ as the argument progresses, restricting instead to ranges such as $\Lambda/2 \leq t \leq 0$ or $\Lambda/4 \leq t \leq 0$).  The first step is to obtain Riemann-von Mangoldt type formulae for the number of zeroes of $H_t$ in an interval such as $[0,T]$ or $[T,T+\alpha]$ where $T \geq 2$ and $0 < \alpha \leq o(T)$.  When $t=0$, we can obtain asymptotics of $\frac{T}{4\pi} \log \frac{T}{4\pi} - \frac{T}{4\pi} + O(\log T)$ and $\frac{\alpha}{4\pi} \log T + o(\log T)$ by the classical Riemann-von Mangoldt formula and a result of Littlewood respectively; this gives good control on the zeroes down to length scales $\alpha \asymp 1$.  For $\Lambda < t < 0$, we were only able to obtain the weaker bounds of $\frac{T}{4\pi} \log \frac{T}{4\pi} - \frac{T}{4\pi} + O(\log^2 T)$ and $\frac{\alpha}{4\pi} \log T + o(\log^2 T)$ respectively down to length scales $\alpha \asymp \log T$, but it turns out that these bounds still (barely) suffice for our arguments; see Section \ref{vm}.  A key input in the proof of the Riemann-von Mangoldt type formula will be some upper and lower bounds for $H_t(x-iy)$ when $y$ is comparable to $\log x$; see Lemma \ref{coarse_H_estimates} for a precise statement.  The main tool used to prove these bounds is the saddle point method, in which various contour integrals are shifted until they resemble the integral for the Gamma function, to which the Stirling approximation may be applied.

It was shown in \cite{csv} that in the region $\Lambda < t \leq 0$, the zeroes $x_j(t)$ of $H_t$ are simple, and furthermore evolve according to the system of ordinary differential equations
\begin{equation}\label{xkt}
 \partial_t x_k(t) = 2 \sum_{j:\, j \neq k} \frac{1}{x_k(t) - x_j(t)}; 
\end{equation}
see Theorem \ref{dynam} for a more precise statement.  One can view this equation as describing the dynamics of a system of ``particles'' $x_j$, in which every pair of particles $x_j, x_k$ experiences a repulsion\footnote{We caution however that the dynamics here are not Newtonian in nature, since \eqref{xkt} prescribes the velocity $\partial_t x_k$ of each particle rather than the acceleration $\partial_t^2 x_k$.  Nevertheless we found the physical analogy to be helpful in locating the arguments used in this paper.} that is inversely proportional to their separation.
By refining the analysis in \cite{csv}, we can obtain a more quantitative lower bound on the gap $x_{j+1}(t) - x_j(t)$ between adjacent ``particles'' (zeroes), in particular establishing a bound of the form
$$ \log \frac{1}{x_{j+1}(t) - x_j(t)} \ll \log^2 j \log\log j$$
for all large $j$ in the range $\Lambda/2 \leq t \leq 0$; see Proposition \ref{gap} for a more precise statement.  While far from optimal, this bound almost allows one to define the \emph{Hamiltonian}
$$ {\mathcal H}(t) \coloneqq \sum_{j,k:\, j \neq k} \log \frac{1}{|x_j(t)-x_k(t)|},$$
although in practice we will have to apply some spatial cutoffs in $j,k$ to make this series absolutely convergent.  For the sake of this informal overview we ignore this cutoff issue for now.  The significance of this quantity is that the system \eqref{xkt} can (formally, at least) be viewed as the gradient flow for the Hamiltonian ${\mathcal H}(t)$.  In particular, there is a formal monotonicity formula
\begin{equation}\label{het} 
\partial_t {\mathcal H}(t) = - 4 E(t)
\end{equation}
where the \emph{energy} $E(t)$ is defined as
$$ E(t) \coloneqq \sum_{j,k:\, j \neq k} \frac{1}{|x_j(t)-x_k(t)|^2}.$$
Again, in practice one needs to apply spatial cutoffs to $j,k$ to make this quantity finite, and one then has to treat various error terms arising from this cutoff, which among other things ``renormalizes'' the summands $\frac{1}{|x_j(t)-x_k(t)|^2}$ so that the renormalized energy vanishes when the zeroes are arranged in the equilibrium state of an arithmetic progression; we ignore these issues for the current discussion. A further formal calculation indicates that $E(t)$ is monotone non-increasing in time (so that ${\mathcal H}(t)$ is formally convex in time, as one would expect for the gradient flow of a convex Hamiltonian).  Exploiting (a variant of) the equation \eqref{het}, we are able to control integrated energies that resemble the quantities $\int_{\Lambda/2}^0 E(t)\ dt$; see first the weak preliminary integrated energy bound in Proposition \ref{energy-weak}, and then the final integrated energy bound in Theorem \ref{prelim}.  By exploiting local monotonicity properties of the energy (and using a pigeonholing argument of Bourgain \cite{bourgain}), we can then obtain good control (a truncated version of) the energy $E(t)$ at time $t=0$, which intuitively reflects the assertion that the ``particles'' $x_j(t)$ are close to local equilibrium at time $t=0$.  This implies that the zeroes of the Riemann zeta function behave locally like an arithmetic progression on the average.  However, this can be ruled out by the existing results on the local distribution of zeroes, such as pair correlation estimates of Montgomery \cite{montgomery}. As it turns out, it will be convenient to make use of a closely related estimate of Conrey, Ghosh, Goldston, Gonek, and Heath-Brown \cite{cgggh}.

It may be possible to use the methods of this paper to also address the generalized Newman conjecture introduced in \cite{stopple}, but we do not pursue this direction here\footnote{Note added in proof: the generalized Newman conjecture has now been established, with a significantly simpler proof than the one given here: see \cite{dobner}.}.

\begin{remark}  It is interesting to compare this with the results in \cite[Theorem 1.14]{kkl}, which show that regardless of the value of $\Lambda$, the zeroes of $H_t$ will be spaced like an arithmetic progression on average for any \emph{positive} $t$.
\end{remark}

\begin{remark}
Added in press: we note that in forthcoming work, Alex Dobner has found a proof that $\Lambda \geq 0$ which avoids the heat equation approach we have used here. Dobner's approach instead relies on a Riemann-Siegel type approximation for $H_t$ in order to demonstrate the existence of zeros off the critical line.  There is also some very intriguing numerical work of Rudolph Dwars (see the comments to {\tt terrytao.wordpress.com/2018/12/28}) that suggest that many of the zeroes of $H_t, t < 0$ away from the critical line organize around deterministic curves.
\end{remark}

\subsection{Acknowledgments}

The first author received partial support from the NSF grant DMS-1701577 and an NSERC grant. The second author is supported by NSF grant DMS-1266164 and by a Simons Investigator Award.  We thank anonymous referees for useful suggestions, and likewise we thank Charles Newman for helpful comments and Alex Dobner for corrections.

\subsection{Notation}

Throughout the rest of the paper, we will assume for sake of contradiction that Newman's conjecture fails:
$$ \Lambda < 0.$$
In particular this implies the Riemann hypothesis (which, as mentioned previously, is equivalent to the assertion $\Lambda \leq 0$).

We will have a number of logarithmic factors appearing in our upper bounds.  To avoid the minor issue of the logarithm occasionally being negative, we will use the modified logarithm
$$ \log_+(x) \coloneqq \log( 2 + |x| )$$
for several of these bounds.  We also use the standard branch of the complex logarithm, with imaginary part in the interval $(-\pi,\pi]$, and the standard branch $z^{1/2} \coloneqq \exp( \frac{1}{2} \log z)$ of the square root, defined using the standard branch of the complex logarithm.

Let $\Lambda < t \leq 0$, then the zeroes of $H_t$ are all real, and symmetric around the origin.  It is a result of Csordas, Smith, and Varga \cite[Corollary 1]{csv} that the zeroes are also distinct and avoid the origin.  Thus we can express the zeroes of $H_t$ as $(x_j(t))_{j \in \Z^*}$, where $\Z^* \coloneqq \Z \backslash \{0\}$ are the non-zero integers, 
$$ 0 < x_1(t) < x_2(t) < \dots,$$
and $x_{-j}(t) = -x_j(t)$ for all $j \geq 1$.

For any real numbers $j_- \leq j_+$, we use $[j_-,j_+]_{\Z^*}$ to denote the discrete interval 
$$ [j_-, j_+]_{\Z^*} \coloneqq \{ j \in \Z^*: j_- \leq j \leq j_+ \}.$$

We use the usual asymptotic notation $X \ll Y$, $Y \gg X$, or $X = O(Y)$ to denote a bound of the form $|X| \leq C Y$ for some absolute constant $C$, and write $X \asymp Y$ for $X \ll Y \ll X$. Note that as $\Lambda$ is also an absolute constant, $C$ can certainly depend on $\Lambda$; thus for instance $|\Lambda| \asymp 1$.  If we need the implied constant $C$ to depend on other parameters, we will indicate this by subscripts, thus for instance $X = O_\kappa(Y)$ denotes the estimate $|X| \leq C_\kappa Y$ for some $C$ depending on $\kappa$. If the quantities $X,Y$ depend on an asymptotic parameter such as $T$, we write $X = o_{T \to \infty}(Y)$ to denote a bound of the form $|X| \leq c(T) Y$, where $c(T)$ is a quantity that goes to zero as $T \to \infty$.

For $X$ and $Y$ depending on an asymptotic parameter $T$, we will also use the notation $X \lessapprox Y$ or $X = \tilde O(Y)$ for $X \ll Y \log^{O(1)} T$ in the last two sections of this paper.

Furthermore, in sums that will appear which depend on a parameter $T$, we say that indices $j,k$ are \emph{nearby}, and write $j \sim_T k$, if one has
$ 0 < |j-k| < (T^2 + |j| + |k|)^{0.1}.$

We will use a marked sum to indicate principle value summation:
$$\sum_j^\prime \cdots = \lim_{J\rightarrow\infty} \sum_{|j|\leq J} \cdots.$$ 
In cases where there is any chance of confusion for the range of summation we record the index being summed and use a colon to indicate its range; e.g. we write $\sum_{j:\, j \neq k}$ to indicate that the summation is over $j$, and $j$ is to not equal $k$ (where $k$ is fixed outside the sum). Semicolons are used to separate additional conditions.

We use the phrase \emph{for almost every $t$} throughout this paper to denote that a relation holds for all $t$ except a set of null Lebesgue measure.

\section{Asymptotics of $H_t$}\label{asymp-sec}

In this section we establish some upper and lower bounds on $H_t(z)$ and its logarithmic derivative $\frac{H'_t}{H_t}(z)$.  We will be able to obtain reasonable upper bounds in the regime where $z = x-iy$ with $y = O(\log_+ x)$, and obtain more precise asymptotics when $y \asymp \log_+ x$ (as long as the ratio $y / \log_+ x$ is large enough); this will be the key input for the Riemann-von Mangoldt type asymptotics in the next section.  More precisely, we show

\begin{lemma}\label{coarse_H_estimates}  Let $z = x - i \kappa \log_+ x$ for some $x \geq 0$ and $0 \leq \kappa \leq C$, and let $\Lambda < t \leq 0$. 
Then one has\footnote{The reader is advised not to take the numerous factors of $\pi$, $\sqrt{2}$, etc. appearing in this section too seriously, as the exact numerical values of these constants are not of major significance in the rest of the arguments.}
\begin{equation} \label{H_bound}
H_t(z) \ll \exp\left( -\frac{\pi x}{8} + O_C(\log_+^2 x)\right).
\end{equation}
Furthermore, there is an absolute constant $C'>0$ (not depending on $C$) such that if $\kappa \geq C'$, then one has the refinement
\begin{equation} \label{H_asymp}
H_t(z) = \exp\left(-\frac{\pi x}{8} + O_C(\log_+^2 x)\right),
\end{equation}
as well as the additional estimate
\begin{equation} \label{logderiv_H_asymp}
\frac{H_t'}{H_t}(z) = \frac{i}{4} \log\left(\frac{iz}{4\pi}\right) + O_C\left(\frac{\log_+ x}{x}\right),
\end{equation}
using the standard branch of the complex logarithm.  
\end{lemma}

\begin{remark} With a little more effort one could replace the hypothesis $\Lambda < t$ here by $-C < t$; in particular (in contrast to the remaining arguments in this paper) these results are non-vacuous when $\Lambda \geq 0$.  However, we will need to assume $\Lambda < t$ in the application of these estimates in the next section, particularly with regards to the proof of \eqref{little_o_bound}.  Our proof methods also allow for a more precise version of the asymptotic \eqref{H_asymp} (as one might expect given the level of precision in \eqref{logderiv_H_asymp}), but such improvements do not seem to be helpful for the rest of the arguments in this paper.  In the $t=0$ case, one can essentially obtain Dirichlet series expansions for $\frac{1}{H_0(z)}$ or $\frac{H'_0}{H_0}(z)$ which allow one to also obtain bounds such as \eqref{H_asymp} or \eqref{logderiv_H_asymp} when the imaginary part of $z$ is much smaller than $\log_+ x$.  However, in the $t<0$ case there does not appear to be any usable series expansions for $\frac{1}{H_t}(z)$ or $\frac{H_t'}{H_t}(z)$ that could be used to prove \eqref{H_asymp} or \eqref{logderiv_H_asymp}.  Instead, we will prove these estimates by computing $H_t(z)$ to a high degree of accuracy, which we can only do when $y$ is greater than or equal to a large multiple of $\log_+ x$ in order to ensure that the series expansions we have for $H_t(z)$ converge rapidly.
\end{remark}

We begin by treating the easy case $t=0$, in which we can exploit the identity \eqref{hoz}.  We have the very crude bound
\begin{equation}\label{zsit}
 \zeta( \sigma + i \tau ) \ll (1+|\tau|)^{O(1)}
\end{equation}
whenever $\sigma \geq 1/2$ and $\tau \in \R$ (this follows for instance from \cite[Theorem 4.11]{titch}).  In the region $\sigma \geq 1/4$, we also have the Stirling approximation (see e.g. \cite[6.1.41]{as})
\begin{equation}\label{stirling}
 \Gamma( \sigma + i\tau ) = \exp\left( \left(\sigma + i\tau - \frac{1}{2}\right) \log (\sigma+i\tau) - (\sigma + i\tau) + \log \sqrt{2\pi} + O\left( \frac{1}{|\sigma+i\tau|} \right) \right),
\end{equation}
where we use the standard branch of the logarithm; in particular
\begin{equation}\label{stir2}
 \Gamma( \sigma + i\tau ) \ll \exp\left( (\sigma - \frac{1}{2}) \log |\sigma+i\tau| - \tau \operatorname{arctan} \frac{\tau}{\sigma} - \sigma \right).
\end{equation}
As $\operatorname{arctan} \frac{\tau}{\sigma} = \frac{\pi}{2} \operatorname{sgn}(\tau) + O( \frac{\sigma}{\sigma + |\tau|} )$, we have in particular that
$$
 \Gamma( \sigma + i\tau ) \ll \exp\left( - \frac{\pi}{2} |\tau| + O( \sigma \log_+( |\sigma| + |\tau| ) ) \right).$$
Inserting these bounds into \eqref{hoz}, \eqref{sas}, we obtain the crude upper bound
\begin{equation}\label{h0-upper}
 H_0(x-iy) \ll \exp\left( - \frac{\pi |x|}{8} +  O( (1+y) \log_+(|x|+y) ) \right)
\end{equation}
for $x \in\R$ and $y \geq 0$.  This gives the $s=0$ case of \eqref{H_bound}.  As is well known, when $\sigma \geq 2$ (say) we can improve \eqref{zsit} to
$$ |\zeta(\sigma+i\tau)| \asymp 1$$
and so we obtain the improvement
$$ H_0(x-iy) = \exp\left( - \frac{\pi |x|}{8} +  O( (1+y) \log_+(|x|+y) ) \right)$$
when $y \geq C' \log_+ x$ (in fact in this case it would suffice to have $y \geq 4$, say).  This gives the $s=0$ case of \eqref{H_asymp}.  Finally, from taking logarithmic derivatives of \eqref{hoz}, \eqref{sas} one has
$$ \frac{H'_0}{H_0}(z) = \frac{i}{2} \left( \frac{1}{s} + \frac{1}{s-1} - \frac{1}{2} \log \pi + \frac{1}{2} \frac{\Gamma'}{\Gamma} \left( \frac{s}{2} \right ) + \frac{\zeta'}{\zeta}(s) \right)$$
where $s \coloneqq \frac{1}{2}  + \frac{iz}{2}$.  From taking log-derivatives of \eqref{stirling} using the Cauchy integral formula, one has the well known asymptotic
$$ \frac{\Gamma'}{\Gamma} \left( \frac{s}{2} \right )  = \log \frac{s}{2} + O\left( \frac{1}{|s|} \right)$$
for the digamma function $\frac{\Gamma'}{\Gamma}$, and from the Dirichlet series expansion $\frac{\zeta'}{\zeta}(s) = - \sum_{n=1}^\infty \frac{\Lambda(n)}{n^s} \ll \sum_{n=2}^\infty \frac{\log n}{n^{\Re s}}$ one can easily establish the bound
$$\frac{\zeta'}{\zeta}(s)  \ll \frac{1}{|s|}$$
in the regime $C' \log_+ x \leq y \leq C \log x$.  Putting all this together, one obtains \eqref{logderiv_H_asymp} in this case.

Henceforth we address the $t<0$ case.
We begin with the proof of the upper bound \eqref{H_bound}.  Here it will be convenient to exploit the fundamental solution for the (backwards) heat equation to relate $H_t$ with $H_0$.
Indeed, for any $t<0$, we have the classical heat equation (or Gaussian) identity
\begin{equation}\label{tzu}
 e^{tu^2} \exp(izu) = \frac{1}{\sqrt{4\pi}} \int_\R e^{-r^2/4} \exp\left(i(z+r|t|^{1/2})u\right)\ dr
\end{equation}
for any complex numbers $z,u$; replacing $z,r$ by $-z,-r$ and averaging we conclude that
$$ e^{tu^2} \cos(zu) = \frac{1}{\sqrt{4\pi}} \int_\R e^{-r^2/4} \cos\left((z+r|t|^{1/2})u\right)\ dr.$$
 Multiplying by $\Phi(u)$, integrating $u$ from $0$ to infinity, and using Fubini's theorem, we conclude that
\begin{equation}\label{htz}
 H_t(z) = \frac{1}{\sqrt{4\pi}} \int_\R e^{-r^2/4} H_0(z+r|t|^{1/2})\, dr.
\end{equation}
Applying \eqref{h0-upper}, the triangle inequality, and the hypothesis $\Lambda < t \leq 0$, we conclude that
\begin{multline*}
H_t(x-iy) \ll \exp\left( - \frac{\pi |x|}{8} +  O( (1+y) \log_+(|x|+y) ) \right)\\
\times \int_\R \exp\left( -\frac{r^2}{4} + O( (1+y + |r|) \log_+ r ) \right)\, dr.
\end{multline*}
Using $(1+|r|) \log_+ r \leq \eps r^2 + O_\eps(1)$ and $y \log_+ r \ll \eps r^2 + O_\eps(y^2)$ for any absolute constant $\eps>0$, we have
$$ -\frac{r^2}{4} + O(|r|) + O( (1+y+|r|) (1+ \log_+ r) ) \leq -\frac{r^2}{8} + O( (1+y)^2 ),$$
thus arriving at the bound
$$ H_t(x-iy) \ll \exp\left( - \frac{\pi |x|}{8} +  O( (1+y) \log_+ |x| + (1+y)^2 ) \right).$$
Since $y = O_C(\log_+ x)$, this gives \eqref{H_bound}.

To prove the remaining two bounds \eqref{H_asymp}, \eqref{logderiv_H_asymp}, it is convenient to cancel off the $t=0$ case that has already been established, and reduce to showing that
\begin{equation} \label{H_asymp-cancel}
\frac{H_t}{H_0}(z) = \exp\left( O_C(\log_+^2 x)\right),
\end{equation}
and
\begin{equation} \label{logderiv_H_asymp-cancel}
\frac{H_t'}{H_t}(z) - \frac{H_0'}{H_0}(z) \ll_C \frac{\log_+ x}{x},
\end{equation}
when $\kappa \geq C'$.  To prove these estimates, the heat equation approach is less effective due to the significant oscillation present in $H_0$.  Instead we will use the method of steepest descent (also known as the saddle point method) to shift contours to where the phase is stationary rather than oscillating.  We allow all implied constants to depend on $C$.  We may assume that $x$ is larger than any specified constant $C''$ (depending on $C$), as the case $x=O_C(1)$ follows trivially from compactness, since the zeroes of $H_t$ for $t \geq \Lambda$ are all real, so that $H_t(z)$ is bounded away from zero in this region of interest.

Now suppose that $z = x-iy$ where $y = \kappa \log_+ x$ for some $C' \leq \kappa \leq C$; in particular $C$ is large since $C'$ is.  As $\Phi$ is even, we may write \eqref{htdef} as
$$ H_t(z) = \frac{1}{2} \int_\R e^{tu^2} \Phi(u) e^{izu}\ du.$$
From \eqref{phidef} and Fubini's theorem (which can be justified when $t<0$) we conclude that
\begin{equation}\label{ht1}
 H_t(z) = \frac{1}{2} \sum_{n=1}^\infty 2\pi^2 n^4 I_t( \pi n^2, 9+y+ix) - 3 \pi n^2 I_t( \pi n^2, 5+y+ix)
\end{equation}
where $I_t(b,\zeta)$ denotes the oscillatory integral
\begin{equation}\label{itz-def}
I_t(b,\zeta) \coloneqq  \int_\R \exp(tw^2 - b e^{4w} + \zeta w) \, dw,
\end{equation}
which is an absolutely convergent integral for $t < 0$ whenever $\Re b > 0$.  

We therefore need to obtain good asymptotics on $I_t(b,\zeta)$ for $b \geq 1$ and $\zeta$ in the region
\begin{equation}\label{om-def}
 \Omega \coloneqq \{ y+ix: x \geq C''; C' \log_+ x \leq y \leq 2C \log_+ x \}.
\end{equation}
Observe that the phase $tw^2 - be^{4w} + \zeta w$ has a stationary point at the origin when $4b = \zeta$.  In general, $4b$ will not equal $\zeta$; however, for any complex number $w_0$ in the strip
\begin{equation}\label{strip}
\left\{ w_0 \in \C: 0 \leq \Im(w_0) < \frac{\pi}{8} \right\};
\end{equation}
we see from shifting the contour in \eqref{itz-def} to the horizontal line $\{ w+w_0: w \in \R \}$ that we have the identity
\begin{equation}\label{scaling}
 I_t(b,\zeta) = \exp( tw_0^2 + \zeta w_0 ) I_t( b e^{4w_0}, \zeta + 2 t w_0 ) 
\end{equation}
whenever $b>0$ (so that $be^{4w_0}$ has positive real part).
We will thus be able to reduce to the stationary phase case $4b=\zeta$ if we can solve the equation 
\begin{equation}\label{4bw}
4b e^{4w_0} = \zeta + 2tw_0
\end{equation}
in the strip \eqref{strip}.  This we do in the following lemma\footnote{One could also write $w_0$ explicitly in terms of the Lambert $W$-function as $w_0 = -\frac{\zeta}{2t} + \frac{1}{4} W( -\frac{8b}{t} \exp( -\frac{2\zeta}{t} ) )$, but we will not use this expression in this paper, and in fact will not explicitly invoke any properties of the $W$-function in our arguments.}:

\begin{lemma}\label{Wfind}  If $b \geq 1$ and $\zeta \in \Omega$, then there exists a unique $w_0 = w_0(b,\zeta)$ in the strip \eqref{strip} such that \eqref{4bw} holds.  Furthermore we have the following estimates:
\begin{itemize}
\item[(i)] $\Re(4be^{4w_0}) \geq 1$.
\item[(ii)] (Precise asymptotic for small and medium $b$) If $\zeta = y+ix$ and $b \leq x \exp( 100 \frac{x^{1/2}}{|t|})$, then 
$$ w_0 = \frac{1}{4} \log \frac{x}{4b} + O^\R\left(\frac{1}{x}\right) + i \left(\frac{\pi}{8} - \frac{y}{4x} - \frac{t \log \frac{x}{4b}}{8 x} + O^\R_C\left(\frac{\log^2_+ x}{x^{3/2}}\right)\right)$$
where the superscript in the $O()$ notation indicates that these quantities are real-valued.
\item[(iii)]  (Crude bound for huge $b$) If $\zeta = y+ix$ and $b > x \exp( \frac{x^{1/2}}{|t|})$, then $\Re w_0$ is negative; in fact we have
$$ -\Re w_0 \geq \frac{1}{8} \log_+ b.$$
\end{itemize}
\end{lemma}

\begin{proof}  The function $w_0 \mapsto 4b e^{4w_0} - 2tw_0$ traverses the graph $\{ a + i (\frac{\pi |t|}{4} + 4b e^{2a/|t|}): a \in \R\}$ on the upper edge $\{ \frac{a}{2|t|} + i \frac{\pi}{8}: a \in \R \}$ of the strip \eqref{strip}, while the lower edge of the strip is of course mapped to the real axis.  Since $|t| \leq \Lambda$ and $C, C'$ are large, the region $\Omega$ lies between these two curves, and so from the argument principle (and observing that the map $w_0 \mapsto 4b e^{4w_0} - 2tw_0$ sends the line segments $\{ - R + i\beta: 0 < \beta < \pi/8\}$ and $\{ R + i\beta: 0 < \beta < \pi/8\}$ well to the left and right of $\zeta$ respectively for $R$ large enough), for every $\zeta \in \Omega$ there exists exactly one $w_0$ in the strip \eqref{strip} such that $4b e^{4w_0} - 2 t w_0 = \zeta$, which is of course equivalent to \eqref{4bw}.  The uniqueness implies that the holomorphic function $w_0 \mapsto 4b e^{4w_0} - 2tw_0$ has non-zero derivative at this value of $w_0$.

Now write $\zeta = y+ix$ as per \eqref{om-def}, and write $w_0 = \alpha+i\beta$ for some $\alpha \in \R$ and $0 < \beta < \pi/8$.   Taking real and imaginary parts in \eqref{4bw} we have the system of equations
\begin{equation}\label{real}
4b e^{4\alpha} \cos 4\beta = y + 2 t \alpha
\end{equation}
and
\begin{equation}\label{imaginary}
4b e^{4\alpha} \sin 4\beta = x + 2 t \beta.
\end{equation}

To prove (i), suppose for contradiction that $\Re(4be^{4w_0}) < 1$, thus
\begin{equation}\label{hypot}
 4b e^{4\alpha} \cos 4\beta \leq 1.
\end{equation}
Since $t, \beta = O(1)$, we see from \eqref{imaginary} that $4 b e^{4\alpha} \sin 4\beta \ll x$, and hence from $\sin^2 4\beta + \cos^2 4\beta = 1$ we have
$$ 4b e^{4\alpha} \ll x$$
and hence (since $b \geq 1$) $\alpha \leq \frac{1}{4} \log_+ x + O(1)$.  In particular $-2t\alpha \leq \frac{|t|}{2} \log_+ x + O(1)$.  Inserting this into \eqref{real} and using \eqref{hypot} one then has
$$ y \leq \frac{|t|}{2} \log_+ x + O(1),$$
which contradicts \eqref{om-def} since $|t| \leq \Lambda$ and $C'$ is large.

Now we show (ii).  From \eqref{imaginary} and $\sin 4\beta \leq 1$, $t,\beta = O(1)$ one has
$$ 4be^{4\alpha} \geq x - O(1)$$
and hence on taking logarithms (and using the fact that $b \geq 1$ and $x$ is large)
\begin{equation}\label{al}
 \alpha \geq \frac{1}{4} \log \frac{x}{4b} - O\left(\frac{1}{x}\right).
\end{equation}
On the other hand, from squaring \eqref{real}, \eqref{imaginary} and summing we have
\begin{equation}\label{4b}
 (4b e^{4\alpha})^2 = (y + 2t\alpha)^2 + (x+2t\beta)^2.
\end{equation}
Crudely bounding $x+2t\beta = O(x)$, $y = O(x)$, $b \geq 1$, and $t = O(1)$ we conclude that
$$ e^{8\alpha} \ll x^2 + \alpha^2 $$
which implies that $\alpha \leq O(\log_+ x)$.  From the hypothesis $b \leq x \exp( 100\frac{x^{1/2}}{|t|})$ and \eqref{al} we also have $\alpha \geq - O( x^{1/2} / t )$, thus $t\alpha \ll x^{1/2}$.  Returning to \eqref{4b} and using $2t\beta = O(1)$ and $y \ll x^{1/2}$ we conclude that
$$ (4b e^{4\alpha})^2 =x^2 + O(x)$$
so on taking square roots
\begin{equation}\label{4bs}
4be^{4\alpha} = x + O(1)
\end{equation}
and hence on taking logarithms we have the matching upper bound
$$ \alpha \leq \frac{1}{4} \log \frac{x}{4b} + O\left(\frac{1}{x}\right)$$
to \eqref{al}.  In particular,
$$ y + 2t\alpha = y + \frac{t \log \frac{x}{4b}}{2} + O\left(\frac{1}{x}\right).$$
Inserting this and \eqref{4bs} into \eqref{real}, we have
$$ \cos 4\beta = \frac{y}{x} + \frac{t \log \frac{x}{4b}}{2x} + O\left(\frac{1}{x^{3/2}}\right) $$
and hence (by Taylor expansion of the arc cosine function)
$$ 4\beta = \frac{\pi}{2} - \frac{y}{x} - \frac{t \log \frac{x}{4b}}{2x} + O_C\left(\frac{\log^2_+ x}{x^{3/2}}\right),$$
giving (ii).

Finally, we prove (iii).  From the identity \eqref{4b} and crudely bounding $y, t \beta = O(x)$ we have
$$ (4b e^{4\alpha})^2 \ll x^2 + t^2 |\alpha|^2 $$
and hence either
$$ e^{-4\alpha} \gg \frac{b}{x}$$
or
$$ e^{-4\alpha} \gg \frac{b}{|t| |\alpha|}.$$
Under the hypothesis $b > x \exp( \frac{x^{1/2}}{|t|})$, so that $1/|t|$ and $x$ are $O(b^{1/10})$ (say), so both options force $- \alpha \geq \frac{1}{8} \log b$ as claimed.
\end{proof}

We combine the above lemma with the following asymptotic.

\begin{lemma}  Let $b$ be a complex number with $\Re b \geq 1$.  Then
\begin{equation}\label{it1}
 I_t( b, 4b ) = \sqrt{\frac{\pi}{8}} \exp( - b ) \left( \frac{1}{\sqrt{b}} + O\left( \frac{1}{|b|^{3/2}} \right) \right)
\end{equation}
using the standard branch of the square root.
\end{lemma}

\begin{proof}  
One could establish this from Laplace's method, but we will instead use the Stirling approximation\footnote{We thank Alex Dobner for pointing out some issues in the original proof of this lemma, and suggesting a repaired proof which is reproduced here.} \eqref{stirling}.  Writing
$$ e^{tw^2} = \int_\R e^{4 i \xi w}\ d\mu(\xi)$$
where $\mu$ is the Gaussian probability measure
$$ d\mu(\xi) \coloneqq \frac{2}{\sqrt{\pi|t|}} e^{-4 \xi^2/|t|} $$
of mean zero and variance $|t|/8$, and applying Fubini's theorem, we obtain
$$ I_t(b,4b) = \int_\R \left(\int_\R \exp(- b e^{4w} + 4(b + i \xi) w)\ dw\right)\ d\mu(\xi).$$
Making the change of variables $r = be^{4w}$ (and contour shifting or analytic continuation) and the definition $\Gamma(s) = \int_0^\infty e^{-r} r^{s-1}\ dr$ of the $\Gamma$ function, we see that
$$ \int_\R \exp(- b e^{4w} + 4(b + i \xi) w)\ dw = \frac{1}{4}
\exp( - (b + i \xi) \log b ) \Gamma( b + i \xi) $$
and hence
$$ I_t(b,4b) = \frac{1}{4} \int_\R \exp( - (b + i \xi) \log b ) \Gamma( b + i \xi)\ d\mu(\xi).$$
We divide integral into regions \(|\xi| \leq 10|t|^{1/2}|b|^{1/2}\) and \(|\xi| > 10|t|^{1/2}|b|^{1/2}\). By applying Stirling's approximation, the integral over the first region becomes
\[\frac{1}{4}\int_{|\xi| \leq 10|t|^{1/2}|b|^{1/2}}{\left(1+O\left(\frac{1}{|b+i\xi|}\right)\right)\frac{\sqrt{2\pi}}{\sqrt{b+i\xi}}\exp((b+i\xi)\log(1+\frac{i\xi}{b})-b-i\xi)\, d\mu(\xi)}.\]
Note we may assume \(|b|\) is sufficiently large so that \(|\xi| < |b|/2\) in this region (the small \(|b|\) case of the lemma follows trivially from compactness), and so we have \(\frac{1}{\sqrt{b+i\xi}}=\left(1+O\left(\frac{|\xi|}{|b|}\right)\right)\frac{1}{\sqrt{b}}\) and \((b+i\xi)\log(1+\frac{i\xi}{b})=i\xi+O\left(\frac{|\xi|^2}{|b|}\right)\). Substituting these expressions into the integrand we get
\[\sqrt{\frac{\pi}{8b}}\exp(-b)\int_{|\xi| \leq 10|t|^{1/2}|b|^{1/2}}{\left(1+O\left(\frac{1+|\xi|^2}{|b|}\right)\right)\, d\mu(\xi)},\]
and now the integral evaluates to \(1+O(\frac{1}{|b|})\). Thus it will suffice to establish the tail bound
\[\int_{|\xi| > 10|t|^{1/2}|b|^{1/2}}{\exp(-(b+i\xi)\log b)\Gamma(b+i\xi)\, d\mu(\xi)} \ll \exp(-\Re(b))|b|^{-3/2}.\]
By applying the triangle inequality and bounding the integrand with
\[|\exp(-(b+i\xi)\log b)| \leq \exp(-\Re(b)\log{|b|}+\frac{\pi}{2}(|b|+|\xi|))\]
and
\[|\Gamma(b+i\xi)| \leq \Gamma(\Re(b)) \leq \exp(\Re(b)\log{|b|}-\Re(b))\]
we get the following upper bound:
\[\exp(-\Re(b)+\frac{\pi}{2}|b|) \frac{2}{\sqrt{\pi |t|}} \int_{|\xi| > 10|t|^{1/2}|b|^{1/2}}{\exp(\frac{\pi}{2}|\xi|-\frac{4\xi^2}{|t|})}.\]
Now again we assume \(|b|\) is large enough so that we have \(\frac{\pi}{2}|\xi|-\frac{4\xi^2}{|t|} \leq -\frac{\xi^2}{|t|}\) for all \(\xi\) in the given region, and hence the integral is bounded above by
\[\exp(-\Re(b)+\frac{\pi}{2}|b|) \frac{2}{\sqrt{\pi |t|}} \int_{|\xi| > 10|t|^{1/2}|b|^{1/2}}{\exp(-\frac{\xi^2}{|t|})} \ll \exp(-\Re(b)-10|b|)\]
(say), and the claim follows.
\end{proof}

From the above two lemmas and \eqref{scaling}, we have the asymptotic
\begin{equation}\label{itb}
 I_t(b,\zeta) = \sqrt{\frac{\pi}{8}} \exp( tw_0^2 - b e^{4w_0} + \zeta w_0  ) \left(\frac{1}{\sqrt{b e^{4w_0}}} + O\left( \frac{1}{|be^{4w_0}|^{3/2}} \right)\right)
\end{equation}
for any $b \geq 1$ and $\zeta \in \Omega$, where $w_0 = w_0(b,\zeta)$ is the quantity in Lemma \ref{Wfind}.

Now we can control the sum \eqref{ht1}.  As before  we assume that $z = x-iy$ where $y = \kappa \log_+ x$ for some $C' \leq \kappa \leq C$.
From \eqref{ht1} one has
\begin{equation}\label{qnt}
 H_t(x-iy) = \frac{1}{2} \sum_{n=1}^\infty Q_{t,n}
\end{equation}
where $Q_{t,n}$ is the quantity
$$ Q_{t,n} \coloneqq 2\pi^2 n^4 I_t( \pi n^2, 9+y+ix) - 3 \pi n^2 I_t( \pi n^2, 5+y+ix).$$
We first consider the estimation of $Q_n$ in the main case when $n$ is not too huge, in the sense that
\begin{equation}\label{nb}
n \leq x \exp\left( 100 \frac{x^{1/2}}{|t|} \right).
\end{equation}
In this case, if we apply Lemma \ref{Wfind}(ii) with $\zeta = 9 + y + ix$ and $b = \pi n^2$ we have that the quantity $w_0 = w_{0,t,n}$ arising in that lemma obeys the asymptotics
\begin{equation}\label{wow}
 w_0 = \frac{1}{4} \log \frac{x}{4\pi n^2} + O^\R\left(\frac{1}{x}\right)+ i \left(\frac{\pi}{8} - \frac{9+y}{4x} - \frac{t \log \frac{x}{4\pi n^2}}{8 x} + O^\R_C\left(\frac{\log^2_+ x}{x^{3/2}}\right)\right) ,
\end{equation}
which when combined with \eqref{4bw}, gives
$$ 4 b e^{4w_0} = ix + O_C( x^{1/2}).$$
In particular, the factor $\frac{1}{\sqrt{b e^{4w_0}}} + O\left( \frac{1}{|be^{4w_0}|^{3/2}} \right)$ in \eqref{itb} can be expressed as
$$ \frac{1}{\sqrt{ix/4}} \left(1 + O_C\left( x^{-1/2}\right)\right),$$
and thus by \eqref{itb}
$$ |I_t( \pi n^2, 9 + y + ix)| =\sqrt{\frac{\pi}{2x}} \exp\left( \Re\left( tw_0^2 - \frac{\zeta}{4} - \frac{tw_0}{2} + \zeta w_0\right) + O_C\left( x^{-1/2} \right)\right)$$
where we have again used \eqref{4bw}.  From \eqref{wow} (and using $t=O(1)$ and $y = O_C(\log_+ x)$ to bound some small error terms), we can calculate the quantity $\Re\left( tw_0^2 - \frac{\zeta}{4} - \frac{tw_0}{2} + \zeta w_0\right)$ to be
$$ \frac{t}{16} \log^2 \frac{x}{4\pi n^2} - \frac{t \pi^2}{64} - \frac{9+y}{4} - \frac{t}{8} \log \frac{x}{4\pi n^2} + \frac{9+y}{4} \log \frac{x}{4\pi n^2} - \frac{\pi x}{8} + \frac{9+y}{4} + \frac{t \log \frac{x}{4\pi n^2}}{8} + O_C\left( x^{-1/2} \right)$$
and thus on cancelling and gathering terms we obtain
$$ |I_t( \pi n^2, 9 + y + ix)| = \left(\frac{x}{4\pi n^2}\right)^{\frac{9+y}{4}} J_t K_{t,n} \exp\left( O_C\left( x^{-1/2} \right) \right)$$
where $J_t = J_t(x)$ and $K_{t,n} = K_{t,n}(x)$ are the positive quantities
\begin{equation}\label{j-def}
 J_t \coloneqq \sqrt{\frac{\pi}{2x}} \exp\left( \frac{t}{16} \log^2\frac{x}{4\pi} - \frac{t \pi^2}{64} - \frac{\pi x}{8} \right)
\end{equation}
and
$$ K_{t,n} \coloneqq \exp\left( -\frac{t}{4} \left(\log \frac{x}{4\pi}\right) \log n + \frac{t}{4} \log^2 n \right).$$
A similar computation gives
$$ |I_t( \pi n^2, 5 + y + ix)| = \left(\frac{x}{4\pi n^2}\right)^{\frac{5+y}{4}} J_t K_{t,n} \exp\left( O_C\left( x^{-1/2} \right) \right)$$
In particular we have the upper bound
$$ Q_{t,n} \ll n^4 \left(\frac{x}{4\pi n^2}\right)^{\frac{9+y}{4}} J_t K_{t,n}$$
for $1 \leq n \leq x \exp(100 \frac{x^{1/2}}{|t|})$, 
and for $n=1$ we have the refinement
\begin{equation}\label{q1-mag}
 |Q_{t,1}| = \left( 2\pi^2 + O_C\left( x^{-1/2} \right)\right) \left(\frac{x}{4\pi}\right)^{\frac{9+y}{4}} J_t.
\end{equation}
Using the crude bound
$$ K_{t,b} \leq \exp\left( -\frac{t}{4} \left(\log \frac{x}{4\pi}\right) \log n \right) \leq n^{-\frac{t}{4} \log x}$$
we conclude that
$$ Q_{t,n} \ll n^{-\frac{1+y}{2} - \frac{t}{4} \log x} |Q_{t,1}|.$$
Since $y \geq C' \log_+ x$, the $2 \leq n \leq x \exp( 100 \frac{x^{1/2}}{|t|} )$ terms sum to $O( |Q_{t,1}| / x )$, thus
$$ \sum_{n \leq x \exp( 100 \frac{x^{1/2}}{|t|}) } Q_{t,n} = \left(1 + O_C\left( \frac{1}{x} \right)\right) Q_{t,1}$$
Also, from \eqref{j-def} we have
$$ |Q_{t,1}| \asymp \left( \frac{x}{4\pi} \right)^{\frac{9+y}{4}} J_t = \exp\left( -\frac{\pi x}{8} + O_C( \log_+^2 x ) \right).$$
Thus, to finish the proof of \eqref{H_asymp} (or \eqref{H_asymp-cancel}), one just needs to show that the tail $\sum_{n > x \exp( 100 \frac{x^{1/2}}{|t|} )} Q_{t,n}$ is negligible compared with the main term $Q_{t,1} = \exp( -\frac{\pi x}{8} + O_C( \log_+^2 x ))$.  Suppose now that $n > x \exp( 100 \frac{x^{1/2}}{|t|} )$.  If we now apply Lemma \ref{Wfind}(iii) with $\zeta = 9 + y + ix$ and $b = \pi n^2$, and write $w_0 = \alpha+i\beta$ with $0 < \beta < \pi/8$, we have that $\alpha$ is negative with
$$ - \alpha \geq \frac{1}{8} \log n,$$
while from \eqref{itb} and \eqref{4bw} (and Lemma \ref{Wfind}(i)) we have
\begin{align*}
I_t( \pi n^2, 9 + y + ix ) &\ll 
\exp( \Re( tw_0^2 - \frac{\zeta}{4} - \frac{tw_0}{2} + \zeta w_0) ) \\
&\ll \exp( -|t| |\alpha|^2 - \frac{|t| |\alpha|}{2} +  O_C( \log_+^2 x ) ). 
\end{align*}
Similarly for $I_t( \pi n^2, 5 + y + ix )$.  Since $\log n \geq 100 \frac{x^{1/2}}{|t|}$, we have $|\alpha| \geq 10 \frac{x^{1/2}}{|t|}$ and thus
$$ |t| |\alpha|^2 \geq 10 x^{1/2} \log n.$$
In particular, $n^4 \exp( - |t| |\alpha|^2 ) \ll \exp( - 9 x^{1/2} \log n )$ and thus
$$ Q_{t,n} \ll \exp( - 8 x^{1/2} \log n + O_C( \log_+^2 x) )$$
(say). Summing, we conclude that
$$ \sum_{n >  x \exp( 100 \frac{x^{1/2}}{|t|} )} Q_{t,n} \ll \exp( - 100 x / |t| )$$
(say), which is certainly $O( |Q_1| / x )$.  Inserting these bounds into \eqref{ht1}, we conclude that
$$ H_t( x-iy ) = \left(\frac{1}{2} + O_C\left( \frac{\log^2_+ x}{x} \right)\right) Q_{t,1},$$
which already gives \eqref{H_bound}. Sending $t$ to $0$, taking absolute values, and then dividing using \eqref{q1-mag} and \eqref{j-def}, we obtain after cancelling all the $t$-independent terms that
$$ \left| \frac{H_t}{H_0} \right| (x-iy) = \left(1 + O_C\left( \frac{\log^2_+ x}{x} \right)\right) \exp\left( \frac{t}{16} \log^2\frac{x}{4\pi}  - \frac{t \pi^2}{64} \right).$$
Since the ratio $\frac{H_t}{H_0}$ is holomorphic in the region of interest, we can thus find a holomorphic branch of $\log \frac{H_t}{H_0}$ for which
$$ \Re \log \frac{H_t}{H_0}(z) - \frac{t}{16} \log^2 \frac{z}{4\pi i} = - \frac{t \pi^2}{64} + O_C\left( \frac{\log^2_+ x}{x} \right)$$
for all $z=x-iy$ in this region.  Varying $x,y$ by $O( \log_+ x)$ (adjusting the constants $C,C',C''$ slightly as necessary) and using the Borel-Carath\'eodory theorem and the Cauchy integral formula, we conclude that
$$  \frac{d}{dz} \left( \log \frac{H_t}{H_0}(z) - \frac{t}{16} \log^2 \frac{z}{4\pi i} \right) = O_C\left( \frac{\log_+ x}{x} \right),$$
which gives \eqref{logderiv_H_asymp-cancel} after a brief calculation.

\section{Riemann-von Mangoldt type formulae}\label{vm}

For any $\Lambda < t \leq 0$, the zeroes of $H_t$ are all real and simple \cite[Corollary 1]{csv}.  For any interval $I \subset \R$, let $N_t(I)$ denote the number of zeroes of $H_t$ in $I$.  The classical Riemann-von Mangoldt formula (see e.g. \cite[Theorem 9.4]{titch}), combined with \eqref{hoz}, gives the asymptotic
\begin{equation}\label{nor}
N_0([0,T]) = \Psi(T) + O(\log_+ T)
\end{equation}
for all $T \geq 0$, where we use $\Psi \colon \R^+ \to \R$ to denote\footnote{It is traditional to also insert the lower order term $-\frac{7}{8}$ here, but this term will not be of use in our analysis and will therefore be discarded.  The factors of $4\pi$ are not of particular significance and may be ignored by the reader on a first reading.} the function
\begin{equation}\label{phit}
 \Psi(T) \coloneqq \frac{T}{4\pi} \log \frac{T}{4\pi} - \frac{T}{4\pi}.
\end{equation}
For future reference, we record the derivative of $\Psi$ as
\begin{equation}\label{phi-deriv}
\Psi'(T) = \frac{1}{4\pi} \log \frac{T}{4\pi},
\end{equation}
in particular $\Psi$ is increasing for $T > 4\pi$.
Applying \eqref{nor} with $T$ replaced by $T+\alpha$ and subtracting, we conclude from the mean value theorem that
\begin{equation}\label{nor-2}
N_0([T,T+\alpha]) = \frac{\alpha\log_+ T}{4\pi}  + O( \log_+ T)
\end{equation}
for all $T \geq 0$ and $0 \leq \alpha \leq C$ for any fixed $C$, where the implied constants in the asymptotic notation are allowed to depend on $C$.  Because we are assuming the Riemann hypothesis (and hence the Lindel\"of hypothesis), one can improve this latter bound\footnote{Indeed, on the Riemann hypothesis one can improve the error term to $O\left( \frac{\log_+ T}{\log_+ \log_+ T}\right)$; see \cite[Theorem 14.13]{titch}.  However, we will not need this further refinement in this paper.} to 
\begin{equation}\label{nor-3}
N_0([T,T+\alpha]) = \frac{\alpha\log_+ T}{4\pi}  + o_{T \to \infty}( \log_+ T),
\end{equation}
a result of Littlewood (see \cite[Theorem 13.6]{titch}).  A key input in these bounds is a lower bound on $|\zeta(s)|$ when $\mathrm{Re}(s)$ is somewhat large, e.g. between $2$ and $3$; this is easily obtained through the Dirichlet series identity $\frac{1}{\zeta(s)} = \sum_{n=1}^\infty \frac{\mu(n)}{n^s}$ that is valid in this region.

Define the \emph{classical location} $\xi_j$ of the $j^{\operatorname{th}}$ zero for $j \geq 1$ to be the unique quantity in $(1,+\infty)$ solving\footnote{As with the quantity $w_0$ introduced in Lemma \ref{Wfind}, one could express $\xi_j$ explicitly in terms of the Lambert $W$ function if desired as $\xi_j = 4 \pi e \exp( W( j/e) )$, but we will not use this relation in this paper.} the equation
\begin{equation}\label{xij}
\Psi(\xi_j) = j,
\end{equation}
and extend this to negative $j$ by setting $\xi_{-j} \coloneqq -\xi_j$.  Clearly the $\xi_j$ are increasing in $j$.  For future reference we record the following bounds on the $\xi_j$:

\begin{lemma}[Spacing of the classical locations]  \ 
\begin{itemize}
\item[(i)]  For any $j \geq 1$, one has
\begin{equation}\label{xia}
 \xi_j = (1 + o_{|j| \to \infty}(1)) \frac{4\pi j}{\log_+ j}
\end{equation}
In particular, $\xi_j \asymp \frac{j}{\log_+ j}$ and $\log_+ \xi_j \asymp \log_+ j$.  
\item[(ii)] For any $j,k \in \Z^*$, one has
\begin{equation}\label{xidif}
|\xi_k - \xi_j| \asymp \frac{|k-j|}{\log_+ (|\xi_j| + |\xi_k|)}.
\end{equation}
\item[(iii)] If $1 \leq j \asymp k$, then one has the more precise approximation
\begin{equation}\label{add-2}
 \xi_k - \xi_j = \frac{4\pi  (k-j)}{\log \xi_j} + O\left( \frac{|k-j|^2}{j \log^2 \xi_j} \right).
\end{equation}
Of course, the implied constant in the error term in \eqref{add-2} can depend on the implied constants in the hypothesis $j \asymp k$.
\end{itemize}
\end{lemma}

\begin{proof} If $j \geq 1$, then from \eqref{phit} one has
\begin{equation}\label{xijb}
 \xi_j \log_+ \xi_j = (1 + o_{j \to \infty}(1)) 4\pi j
\end{equation}
which implies that $j^{1/2} \ll \xi_j \ll j$ (say), which implies that $\log_+ \xi_j \asymp \log_+ j$; substituting this back into \eqref{xijb}
yields
$$ \xi_j \asymp \frac{j}{\log_+ j}.$$
This in turn implies that $\log_+ \xi_j = (1 + o_{j \to \infty}(1)) \log_+ j$, and using \eqref{xijb} one last time gives \eqref{xij}. 

Now we obtain (ii).  If $j,k$ have opposing sign, then \eqref{xidif} follows from \eqref{xia}, so by symmetry we may assume that $j,k$ are both positive.  If $j$ is much larger than $k$ or vice versa, then the bound \eqref{xidif} follows from \eqref{xia} and the triangle inequality, so we may now restrict attention to the case $1 \leq j \asymp k$.  The estimates \eqref{xidif} and \eqref{add-2} are trivial for $j=O(1)$, so we may assume $j$ to be large. 

From \eqref{xij} we have
$$ \Psi(\xi_k) - \Psi(\xi_j) = k-j$$
and hence by the mean value theorem and \eqref{phi-deriv} we have
$$ \frac{1}{4\pi} \log \frac{T}{4\pi} (\xi_k - \xi_j) = k-j$$
for some $T$ between $\xi_k$ and $\xi_j$.  From \eqref{xia} we see that $T \asymp \xi_j$, and so \eqref{xidif} follows.  Furthermore, we can conclude that
$$ T = \xi_j + O( |\xi_k-\xi_j| ) = \xi_j + O\left( \frac{|k-j|}{\log \xi_j} \right) $$
and hence
$$ \log T = \log \xi_j + O\left( \frac{|k-j|}{\xi_j \log \xi_j} \right) = \log \xi_j + O\left( \frac{|k-j|}{j} \right)$$
and
$$ \frac{1}{\log T} = \frac{1}{\log \xi_j} + O\left( \frac{|k-j|}{j \log^2 \xi_j} \right)$$
giving \eqref{add-2}.
\end{proof}

Applying \eqref{nor} to $T = x_j(0)$ for some $j \geq 1$, we conclude in particular that
$$ \Psi( x_j(0) ) - \Psi(\xi_j) = O(\log_+ x_j(0)).$$
From \eqref{phi-deriv} and the mean value theorem\footnote{One may wish to treat the bounded case $j=O(1)$ separately, to avoid the minor issue that $\Psi(T)$ becomes decreasing for $T<1$.} we conclude that
\begin{equation}\label{xxj}
 x_j(0) = \xi_j + O(1) 
\end{equation}
for all $j \geq 1$, and hence for all $j \in \Z^*$ by symmetry.  In particular, from \eqref{xia} and the fact that $x_1(0) > 0$ we conclude that
$$ x_j(0) \asymp \frac{j}{\log_+ \xi_j} \asymp \frac{j}{\log_+ j}  $$
for all $j \geq 1$.

In a similar vein, if $1 \leq j < k \leq j + \log_+ j$, then from applying \eqref{nor-3} with $T = x_j(0)$ and $\alpha$ equal to (or slightly less than) $x_k(0) - x_j(0)$, we have
$$ k-j = \frac{x_k(0) - x_j(0)}{4\pi} \log_+ \xi_j + o_{j \to \infty}( \log_+ \xi_j )$$
and hence
$$ x_k(0) - x_j(0) = \frac{4\pi (k-j)}{\log_+ \xi_j}  + o_{j \to \infty}(1).$$
Informally, this asserts that the zeroes $x_j(0)$ behave like an arithmetic progression of spacing $\frac{4\pi}{\log_+ \xi_j}$ at spatial scales between $o(1)$ and $1$.  (In fact, when combined with \eqref{xxj} and \eqref{add-2}, we see that this behavior persists for all scales between $o(1)$ and $o(\xi_j)$.)

In this section we use the asymptotics on $H_t$ obtained in the previous section to establish analogous, but weaker, bounds for the zeroes $x_j(t)$ of the functions $H_t$, in which we lose an additional logarithm factor in the error estimates.

\begin{theorem}[Riemann-von Mangoldt type formulae]\label{rmt}  Let $\Lambda < t \leq 0$, $T > 0$, and let $0 \leq \alpha \leq C$ for some $C>0$.  Then one has
\begin{equation} \label{Big_O_Bound} 
N_t([0,T]) = \Psi(T) + O(\log^2_+ T)
\end{equation}
and
\begin{equation}\label{little_o_bound}
 N_t([T,T+\alpha \log_+ T]) = \frac{\alpha\log^2_+ T}{4\pi}  + o_{T \to \infty}(\log^2_+ T).
\end{equation}
The decay rate in the $o_{T \to \infty}()$ error term is permitted to depend on $C$ but is otherwise uniform in $\alpha$.
\end{theorem}

Repeating the previous analysis, we conclude

\begin{corollary}[Macroscopic structure of zeroes]\label{macro}  Let $\Lambda < t \leq 0$.  Then one has
\begin{equation}\label{xji}
 x_j(t) = \xi_j + O( \log_+ \xi_j ) 
\end{equation}
for all $j \in \Z^*$; in particular
\begin{equation}\label{xji0}
 x_j(t) \asymp \frac{j}{\log_+ \xi_j} \asymp \frac{j}{\log_+ j}
\end{equation}
for all $j \geq 1$.  We also have
\begin{equation}\label{add}
x_k(t) - x_j(t) = \frac{4\pi(k-j) }{\log_+ \xi_j} + o_{j \to \infty}(\log_+ \xi_j)
\end{equation}
whenever $1 \leq j < k \leq j + \log_+^2 \xi_j$.
\end{corollary}

Informally, this corollary asserts that the zeroes $x_j(t)$ behave like an arithmetic progression of spacing $\frac{4\pi}{\log_+ \xi_j}$ at spatial scales between $o(\log_+ \xi_j)$ and $o(\xi_j)$.  This level of spatial resolution is worse by a factor of $\log_+ \xi_j$ than what one can achieve for $x_j(0)$, but will still (barely) be enough for our applications.  We remark that a significantly sharper estimate (with an error term of just $O(1)$ in the analog of \eqref{Big_O_Bound}) is available for any fixed $t>0$; see\footnote{Added in press: even sharper estimates have recently been obtained in \cite[Theorem 1.5]{polymath}.} \cite[Theorem 1.4]{kkl}.

We now turn to the proof of the two bounds in Theorem \ref{rmt}. 

\begin{proof}[Proof of \eqref{Big_O_Bound}]
We make use of the argument principle in exactly the same manner as in the classical proof of the Riemann-von Mangoldt formula. By perturbing $T$ slightly if necessary, we may assume that $T$ is not a zero of $H_t$.  Let $\kappa>0$ be a sufficiently large absolute constant. Then the argument principle yields
$$
N_t([0,T]) = \frac{1}{2\pi i} \int_\Gamma \frac{H_t'}{H_t}(z)\, dz,
$$
where $\Gamma$ is the counterclockwise contour carved out by a straight line from $i \kappa \log_+ 0 = i\kappa \log 2$ to $- i \kappa \log_+ 0 = -i \kappa \log 2$, then along the curve $\Gamma_I$ parameterized by $x - i \kappa \log_+ x$ for $x \in [0,T]$, then along the line $\Gamma_{II}$ from $T - i \kappa \log_+ T$ to $T$, then along the vertical line conjugate to $\Gamma_{II}$ and the curve conjugate to $\Gamma_I$, leading back to $i \log 2$.  As the integrand is odd, the, the integral along the line from $i\kappa \log_+ 0$ to $-i\kappa \log_+ 0$ vanishes. Using the symmetry $H_t(\overline{z})  = \overline{H_t(z)}$, we thus have
$$
N_t([0,T]) = \frac{1}{\pi} \Im \left( \int_{\Gamma_{I}} + \int_{\Gamma_{II}}\right) \frac{H_t'}{H_t}(z)\, dz.
$$
From \eqref{logderiv_H_asymp}, \eqref{phi-deriv} one sees that
$$ \frac{1}{\pi} \frac{H_t'}{H_t}(z) = \frac{d}{dz} (\Psi(iz)) +  O\left(\frac{\log_+ x}{x}\right)$$
for $z = x - i \kappa \log_+ x$ on $\Gamma_I$ (extending $\Psi$ to the right half-plane using the standard branch of the logarithm), and hence by the fundamental theorem of calculus
\begin{align*}
\frac{1}{\pi} \Im \int_{\Gamma_I} \frac{H_t'}{H_t}(z)\, dz &= \Im \Psi(iT + \kappa \log_+ T) - \Psi(\log_+ 0) + O(\log_+^2 T) \\
&= \Psi(T) + O(\log_+^2 T).
\end{align*}
On the other hand, if we let $\theta$ be a phase so that $e^{i \theta} H_t(T - i \kappa \log_+ T)$ is real and positive, then
$$
\left| \Im \int_{\Gamma_{II}} \frac{H_t'}{H_t}(z)\, dz \right| \leq \pi (m+1),
$$
where $m$ is the number of zeroes of $\Re e^{i\theta} H_t(z)$ along the contour $\Gamma_{II}$, since the left hand side is the change in $\arg e^{i\theta} H_t(z)$ as $z$ varies over this contour, and for each increment of $\pi$ in the value of $\arg e^{i\theta} H_t(z)$, we must have that $\Re e^{i\theta} H_t(z)$ is zero for some $z$. Note that the number of zeroes of $\Re H_t(z)$ along this contour is the same as the number of zeroes of 
$$
g(s) \coloneqq \tfrac{1}{2}( e^{-i\theta} H_t(is + T) + e^{i\theta} H_t(-is + T))
$$
as $s$ ranges along the line from $0$ to $\kappa \log_+ T$. Hence $m$ is no more than the number of zeroes $m'$ of $g(s)$ in the disc of radius $\kappa \log_+ T$ centered at $\kappa \log_+ T$. 

The count $m'$ we can estimate with Jensen's formula as follows. Let $\mathcal{M}$ be the maximum of $g(s)$ in a disc centered at $\kappa \log_+ T$ of radius $2\kappa \log_+ T$. Using \eqref{H_bound} and the conjugate symmetry of $H_t(z)$, we have
$$
\mathcal{M} \ll e^{-\tfrac{\pi}{8} T + O(\log_+^2 T)}.
$$
Since from \eqref{H_asymp} we have $g(\kappa \log_+ T) = e^{i\theta} H_t(T - i\kappa \log_+ T) = e^{-\tfrac{\pi}{8} T + O(\log_+^2 T)}$, it therefore follows from Jensen's formula (see e.g. \cite[Lemma 6.1]{MoVa}) that
$$
m' \ll \log_+^2 T.
$$
This induces a corresponding bound on the integral of $\frac{H'_t}{H_t}$ over $\Gamma_{II}$ and therefore establishes the claimed estimate for $N_t([0,T])$.
\end{proof}

\begin{proof}[Proof of \eqref{little_o_bound}]  We will use a ``limiting profile argument'' (also known as a ``compactness argument'' or ``normal families argument''), in which one extracts and then studies a limit of suitably rescaled versions of a family of analytic functions to conclude asymptotic information about these functions.  We remark that this sort of argument can also be used in a similar fashion to deduce the Lindel\"of hypothesis from the Riemann hypothesis: see Theorem 1 of {\tt terrytao.wordpress.com/2015/03/01}.

Suppose for contradiction that this claim failed, then there exists a sequence $T_n \to \infty$, 
and bounded sequences $\Lambda < t_n \leq 0$ and $0 \leq \alpha_n \leq C$, as well as an $\eps>0$, such that
\begin{equation}\label{natn}
 \left| N_{t_n}([T_n,T_n+\alpha_n \log_+ T_n]) - \frac{\alpha_n \log_+^2 T_n}{4\pi}\right| > \eps \log^2_+ T_n
\end{equation}
for all $n$.  By perturbing $T_n$ slightly we may assume that $H_{t_n}$ does not vanish at $T_n$ or $T_n + \alpha_n$.

Let $\kappa>0$ be a sufficiently large absolute constant. By the hypothesis $\Lambda < t_n$, the function $H_{t_n}$ has no zeroes in the lower half-plane.  Thus we can define holomorphic functions $F_n$ on the lower half-plane by the formula
$$ F_n(z) \coloneqq \frac{1}{\log^2_+ T_n} \log \frac{H_{t_n}( T_n + z \log_+ T_n )}{H_{t_n}( T_n - i \kappa \log_+ T_n )} $$
with the branch of the logarithm chosen so that $F_n(-i\kappa) = 0$.  From \eqref{H_bound} we see that the $F_n$ are uniformly bounded on any compact subset of the lower half-plane.  Thus, by Montel's theorem (see \cite[Sec 3.2]{StSh}), we may pass to a subsequence and assume that the $F_n$ converge locally uniformly to a holomorphic function $F$ on the lower half-plane; since the $F_n$ all vanish on $-i\kappa$, $F$ does also.  Then by the Cauchy integral formula, the derivatives
$$ F'_n(z) = \frac{1}{\log_+ T_n} \frac{H'_{t_n}}{H_{t_n}}( T_n + z \log_+ T_n) $$
converge locally uniformly to $F'$.  Comparing this with \eqref{logderiv_H_asymp}, we conclude that
$$ F'(z) = \frac{1}{4}$$
whenever the imaginary part of $z$ is sufficiently large and negative.  By unique continuation, we thus have $F'(z)=\frac{1}{4}$ for all $z$ in the lower half-plane; as $F$ vanishes on $-i\kappa$, we thus have
$$ F(z) = \frac{z+i\kappa}{4}$$
on the lower half-plane.  Since $F_n$ converges locally uniformly to $F$, we conclude that
\begin{equation}\label{hatn}
 H_{t_n}( T_n + z \log_+ T_n ) = H_{t_n}( T_n - i \kappa \log_+ T_n ) \exp\left( \frac{z + i\kappa + o_{n \to \infty}(1)}{4} \log^2_+ T_n \right)
\end{equation}
uniformly for $z$ in a compact subset of the lower half-plane.  Similarly, since $F'_n$ converges locally to $F$, we have
\begin{equation}\label{htn}
 \frac{H'_{t_n}}{H_{t_n}}( T_n + z \log_+ T_n ) = \frac{1 + o_{n \to \infty}(1)}{4}  \log_+ T_n
\end{equation}
uniformly for $z$ in a compact subset of the lower half-plane.

Let $\delta>0$ be a small constant. As in the proof of \eqref{Big_O_Bound}, we can use the argument principle (and a rescaling) to write
$$ N_{t_n}([T_n, T_n + \alpha_n \log_+ T_n])
= \frac{\log_+ T_n}{\pi} \Im \left( \int_{\Gamma_{I,n}} + \int_{\Gamma_{II,n}} + \int_{\Gamma_{III,n}} \right) \frac{H_{t_n}'}{H_{t_n}}(T_n + z \log_+ T_n)\, dz,$$
where $\Gamma_{I,n}$, $\Gamma_{II,n}$, $\Gamma_{III,n}$ trace the line segments from $0$ to $-i\delta$, from $-i\delta$ to $\alpha_n-i\delta$, and from $\alpha_n-i\delta$ to $\alpha_n$ respectively.  By \eqref{htn}, the contribution of the $\Gamma_{II,n}$ integral is $\frac{\alpha + o_{n \to \infty}(1) + O(\delta)}{4\pi} \log_+^2 T_n$ (we allow the decay rate in the $o_{n \to \infty}(1)$ errors to depend on $\delta$).  Using the Jensen formula argument used to prove \eqref{Big_O_Bound}, we see that the contribution of the $\Gamma_{I,n}$ integral is bounded in magnitude by
$$ \ll \int_0^1 \log |g_n( \delta + 2\delta e^{2\pi i \alpha})| - \log |g_n(\delta)|\ d\alpha$$
where
$$ g_n(s)  
\coloneqq \tfrac{1}{2}\left( e^{-i\theta_n} H_{t_n}(T_n + is \log_+ T_n) + e^{i\theta_n} H_{t_n}(T_n - is \log_+ T_n)\right )$$
and the phase $\theta_n$ is chosen so that $e^{i\theta_n} H_{t_n}(T_n - i\delta \log_+ T_n)$ is real and positive.
Applying \eqref{hatn} (and the functional equation $H_{t_n}(\overline{z}) = \overline{H_{t_n}(z)}$) when $|\mathrm{Im}(z)| \geq \sqrt{\delta}$ (say), and \eqref{H_bound} (and the functional equation) otherwise, we conclude that the $\Gamma_{I,n}$ integral is equal to
 $\left(o_{n \to \infty}(1) + O(\sqrt{\delta})\right) \log_+^2 T_n$.  Similarly for the $\Gamma_{III,n}$ integral.  Taking $\delta$ to be sufficiently small and $n$ sufficiently large, we contradict \eqref{natn}.
\end{proof}

\section{Dynamics of zeroes}

As remarked in the introduction, the functions $H_t$ solve a backwards heat equation.  As worked out in \cite{csv}, this induces a corresponding dynamics on the zeroes $x_j$ of $H_t$:

\begin{theorem}[Dynamics of zeroes]\label{dynam} For $\Lambda < t \leq 0$, the zeroes $x_j(t)$ depend in a continuously differentiable fashion on $t$ for each $j$, with the equations of motion
\begin{equation}\label{ode}
 \partial_t x_k(t) = 2 \sum_{j:\, j \neq k}^{\prime} \frac{1}{x_k(t) - x_j(t)} 
\end{equation}
for $k \in \Z^*$ and $\Lambda < t \leq 0$, where recall the tick denotes principal value summation over $j \in \Z^*$ (which will converge thanks to \eqref{xji}, \eqref{xia}).
\end{theorem}

\begin{proof}  This follows from \cite[Lemma 2.4]{csv} (the continuity of the derivative following for instance from \cite[Lemma 2.1]{csv}).  \end{proof}

Informally, the ODE \eqref{ode} indicates that the zeroes $x_k(t)$ will repel each other as one goes forward in time.  On the other hand, if the $x_k(t)$ are arranged (locally, at least) in an arithmetic progression, then the ODE \eqref{ode} suggests that the zeroes will be in equilibrium.  If the $x_k$ are not arranged in an arithmetic progression, and instead have some fluctuation in the spacing between zeroes, then heuristically the ODE \eqref{ode} suggests that the zeroes would move away from the more densely spaced regions and towards more sparsely spaced regions, thus converging towards the equilibrium of an arithmetic progression.  This is the intuition behind the convergence to local equilibrium mentioned in the introduction.

One can estimate the speed of this local convergence to equilibrium by the following heuristic calculation.  Consider the zeroes in a region $[T,T+\alpha]$ of space, where $T > 0$ is large and $\alpha$ is reasonably small (e.g. $\alpha = O(\log_+ T)$).  From Theorem \ref{rmt} (or \eqref{xji}, \eqref{xia}), we see that we expect about $\frac{\alpha}{4\pi} \log T$ zeroes in this interval, with an average spacing of $\frac{4\pi}{\log_+ T}$.  Suppose for sake of informal discussion that there is some moderate fluctuation in this spacing, for instance suppose that the left half of the interval contains about $1.5 \frac{\alpha}{8 \pi} \log T$ zeroes and the right half contains only about $0.5 \frac{\alpha}{8\pi} \log_+ T$ zeroes.  Then a back of the envelope calculation suggests that for $x_k(t)$ near the middle of this interval, the right-hand side of \eqref{ode} would be positive and have magnitude $\asymp \frac{\alpha \log_+ T}{\alpha} = \log_+ T$.  Since the length of the interval is $\alpha$, one may then predict that the time needed to relax to equilibrium is about $\alpha / \log_+ T$.  Since we can flow for time $|\Lambda| \asymp 1$, one would expect to attain equilibrium at the final time $t=0$ if the initial length scale $\alpha$ of the fluctuation obeys the bound $\alpha = o_{T \to \infty}(\log_+ T)$.  Happily, this upper bound is precisely what the asymptotic \eqref{add} gives, so we heuristically expect to (barely) be able to establish local equilibrium at time $t=0$.

Of course, one has to make this intuition more precise.  Our strategy for doing so involves exploiting\footnote{This strategy was loosely inspired by the work of Erd\H{o}s, Schlein, and Yau \cite{esy} exploiting the Hamiltonian structure of Dyson Brownian motion to obtain local convergence to equilibrium, since the equations for Dyson Brownian motion resemble that in \eqref{ode} (but with an additional Brownian motion term).  Indeed, Dyson Brownian motion is the diffusion related to the Gibbs measure $\frac{1}{\mathcal Z} e^{-\beta \mathrm{H}}$ for the Hamiltonian studied here.} the formal gradient flow structure of the ODE \eqref{ode}.  Indeed, one may formally write \eqref{ode} as the gradient flow
$$ \partial_t x_k(t) = - \partial_{x_k} \mathrm{H}( (x_j(t))_{j \in \Z^*} ),$$
where $\mathrm{H}$ is the formal ``Hamiltonian''
$$ \mathrm{H}( (x_j)_{j \in \Z^*} ) \coloneqq \sum_{j,k \in \Z^*:\, j \neq k} \log \frac{1}{|x_k - x_j|}$$
where we ignore for this non-rigorous discussion the fact that the series defining $\mathcal{H}$ is not absolutely convergent.  The Hamiltonian is convex, so one expects the quantity
$${\mathcal H}(t) \coloneqq {\rm H}( (x_j(t))_{j \in \Z^*} ) = \sum_{j,k \in \Z^*:\, j \neq k} H_{jk}(t)$$ 
to be decreasing and convex in time, and for the state $(x_j(t))_{j \in \Z^*}$ to converge to a critical point of the Hamiltonian, where
\begin{equation}\label{formal_e}
H_{jk}(t) \coloneqq \log \frac{1}{|x_j(t) - x_k(t)|}
\end{equation}
denotes the Hamiltonian interaction between $x_j(t)$ and $x_k(t)$.  Indeed, a formal calculation using \eqref{ode} yields the identity
$$ \partial_t {\mathcal H}(t) = - 4 E(t)$$
where $E$ is the ``energy''
$$ E(t) \coloneqq \sum_{k,k' \in \Z^*:\, k \neq k'} E_{kk'}(t)$$
and
\begin{equation}\label{ekj-def}
 E_{kk'}(t) \coloneqq \frac{1}{|x_k(t) - x_{k'}(t)|^2}
\end{equation}
denotes the ``interaction energy'' betwen $x_k(t)$ and $x_{k'}(t)$, and we once again ignore the issue that the series is not absolutely convergent.  A further formal calculation using \eqref{ode} again eventually yields
$$ \partial_t E(t) = - 2 \sum_{k,k'\in \Z^*:\, k \neq k'} \left( \frac{2}{|x_k(t) - x_{k'}(t)|^2} - \sum_{k'' \in \Z^*:\, k''\neq k,k'} \frac{1}{(x_{k''}(t) - x_k(t)) (x_{k''}(t)-x_{k'}(t))} \right)^2 $$
suggesting that ${\mathcal H}(t)$ and $E(t)$ are decreasing and that ${\mathcal H}(t)$ is convex, as claimed.

In order to deal with the divergence of the infinite series appearing above, we will need to truncate the Hamiltonian and energy before differentiating them.  The following lemma records some of the identities that arise when doing such truncations:

\begin{lemma}[Identities]\label{ident}  For brevity, we suppress explicit dependence on the time parameter $t \in (\Lambda,0]$.  Let $K \subset \Z^*$ be a finite set of some cardinality $|K|$.  All summation indices such as $i,j,k$ are assumed to lie in $\Z^*$.
\begin{itemize}
\item[(i)] (Dynamics of a gap, cf. \cite[Lemma 2.4]{csv}) If $j,k \in \Z^*$ are distinct, then
$$ \partial_t (x_k-x_j) = \frac{4}{x_k-x_j} - 2 (x_k - x_j) \sum_{i:\, i \neq k,j} \frac{1}{(x_i-x_k)(x_i-x_j)}.$$
\item[(ii)] (Cross-energy inequality, cf. \cite[Lemma 2.5]{csv}) One has
$$ \partial_t \sum_{k \in K; j \not \in K} E_{jk} \geq - \sum_{k \in K; j \not \in K} \frac{8}{(x_k-x_j)^4}$$
in the weak sense that
$$ \sum_{k \in K; j \not \in K} E_{jk}(t_2) - E_{jk}(t_1) \geq - \int_{t_1}^{t_2} \sum_{k \in K; j \not \in K} \frac{8}{(x_k-x_j)^4}(t)\, dt$$
whenever $\Lambda < t_1 < t_2 \leq 0$.
\item[(iii)] (Energy identity) One has
\begin{align*}
 \partial_t \sum_{k,k' \in K:\, k \neq k'} E_{kk'} &= \sum_{\substack{j \not \in K\\ k,k' \in K:\, k \neq k'}} \frac{4}{(x_k-x_{k'})^2 (x_k-x_j)(x_{k'}-x_j)} \\
&\quad - 2 \sum_{k,k' \in K:\, k \neq k'} \left( \frac{2}{(x_k-x_{k'})^2} - \sum_{k'' \in K:\, k'' \neq k,k'} \frac{1}{(x_{k''}-x_k)(x_{k''}-x_{k'})} \right)^2.
\end{align*}
\item[(iv)]  (Virial\footnote{The terminology here is in analogy with the virial identity in $N$-body classical gravitational physics; see e.g., \cite[Exercise 1.48]{tao-book}.} identity) One has
$$
 \partial_t \sum_{k,k' \in K:\, k \neq k'} (x_k-x_{k'})^2 = 4|K|^2 (|K|-1) - \sum_{k,k' \in K:\, k \neq k'} (x_k-x_{k'})^2 \sum_{j \not \in K} \frac{4}{(x_k-x_j)(x_{k'}-x_j)}
$$
\item[(v)]  (Hamiltonian identity) One has
$$
\partial_t \sum_{k,k' \in K:\, k \neq k'} H_{kk'} = -4 \sum_{k,k' \in K:\, k \neq k'} E_{kk'} + 2 \sum_{\substack{j \not \in K \\ k,k' \in K:\, k \neq k'}} \frac{1}{(x_j-x_k)(x_j-x_{k'})}.$$
\end{itemize}
\end{lemma}

A key point in the identities (iii), (iv), (v) is that if one ignores the ``cross terms'' involving interactions between indices in $K$ (representing some ``local subsystem'' of particles) and indices outside of $K$ (representing the ``environment'' that that subsystem interacts with), the right-hand side has a definite sign (negative in the case of (iii) and (v), and positive in the case of (iv)).  This gives a number of useful ``monotonicity formulae'' as long as cross terms are under control.  As discussed above, many of these various monotonicity formulae reflect the formal convexity properties of the Hamiltonian ${\mathcal H}$.  With more effort one can obtain a precise formula for the defect in the inequality in (ii); see \cite[Lemma 2.5]{csv}.

\begin{proof}  From \eqref{ode} one has
$$ \partial_t x_k - \partial_t x_j = \frac{2}{x_k-x_j} - \frac{2}{x_j-x_k} + \sum_{i:\, i\neq k,j} \frac{2}{x_k-x_i} - \frac{2}{x_k-x_j}$$
which gives (i). Note that the series is now absolutely convergent thanks to \eqref{xji}, \eqref{xia}.

Now we prove (ii).  By monotone convergence, it suffices to show that
$$ \sum_{\substack{k \in K\\ j \in [-R,R]_{\Z^*} \backslash K}} E_{jk}(t_2) - E_{jk}(t_1) \geq - \int_{t_1}^{t_2} \sum_{\substack{k \in K \\ j \in [-R,R]_{\Z^*} \backslash K}} \frac{8}{(x_k-x_j)^4}(t)\, dt$$
for all $\Lambda < t_1 \leq t_2 \leq 0$ and all sufficiently large $R$.  By the fundamental theorem of calculus, it suffices to show that
$$ \partial_t \sum_{\substack{k \in K\\ j \in [-R,R]_{\Z^*} \backslash K}} E_{jk} \geq - \sum_{\substack{k \in K \\ j \in [-R,R]_{\Z^*} \backslash K}} \frac{8}{(x_k-x_j)^4}.$$
we can expand the left-hand side as
$$ - 2 \sum_{\substack{k \in K \\ j \in [-R,R]_{\Z^*} \backslash K}} \frac{\partial_t (x_k - x_j)}{(x_k-x_j)^3}$$
which by (i) becomes
$$ - \sum_{\substack{k \in K\\ j \in [-R,R]_{\Z^*} \backslash K}} \frac{8}{(x_k-x_j)^4} + 4 \sum_{\substack{k \in K \\ j \in [-R,R]_{\Z^*} \backslash K \\ i:\, i \neq j,k}} \frac{1}{(x_k-x_j)^2 (x_i-x_k)(x_i-x_j)}$$
and so it will suffice to show that
$$ \sum_{\substack{k \in K \\ j \in [-R,R]_{\Z^*} \backslash K \\ i:\, i \neq j,k}} \frac{1}{(x_k-x_j)^2 (x_i-x_k)(x_i-x_j)} \geq 0.$$
If $R$ is large enough that $[-R,R]_{\Z^*}$ contains $k$, we can split this sum into three parts, depneding on whether $i \in K$, $i \in [-R,R]_{\Z^*} \backslash K$, or $i \not \in [-R,R]_{\Z^*}$. The contribution of the case $i \in K$ can be rewritten as
$$ \sum_{\substack{j \not \in K \\ k,k' \in K:\, k \neq k'}} \frac{4 (x_{k'}-x_j)}{(x_k-x_j)^2 (x_{k'}-x_j)^2 (x_k-x_{k'})}$$
which equals 
$$\sum_{\substack{j \in [-R,R]_{\Z^*} \\ k,k' \in K:\, k \neq k'}} \frac{2}{(x_k-x_j)^2 (x_{k'}-x_j)^2}$$
after symmetrising in $k$ and $k'$, which is clearly non-negative.  Similarly the contribution of the case $i \in [-R,R]_{\Z^*} \backslash K$ is 
$$\sum_{\substack{k \in K \\ j,j' \in [-R,R]_{\Z^*} \backslash K:\, j \neq j'}} \frac{2}{(x_k - x_j)^2 (x_k-x_{j'})^2},$$ 
which is also clearly non-negative.  Finally, for $i \not \in [-R,R]_{\Z^*}$, all summands are already non-negative.  This gives (ii).

For (iii), we can similarly expand the left-hand side as
$$ - 2 \sum_{k,k' \in K:\, k \neq k'} \frac{\partial_t(x_k-x_{k'})}{(x_k-x_{k'})^3}$$
which by (i) becomes
$$ - \sum_{k,k' \in K:\, k \neq k'} \frac{8}{(x_k-x_{k'})^4} + 4 \sum_{\substack{k,k' \in K: k \neq k'\\ i:\, i \neq k,k'}} \frac{1}{(x_k-x_{k'})^2 (x_i-x_k)(x_i-x_{k'})}.$$
To prove (iii), it thus suffices to establish the identity
\begin{align*}
&\sum_{k,k' \in K:\, k \neq k'} \left( \frac{2}{(x_k-x_{k'})^2} - \sum_{k'' \in K:\, k'' \neq k,k'} \frac{1}{(x_{k''}-x_k)(x_{k''}-x_{k'})} \right)^2\\
&\quad = \sum_{k,k' \in K:\, k \neq k'} \frac{4}{(x_k-x_{k'})^4} - 2 \sum_{k,k',k'' \in K:\, k, k', k''\ \mathrm{distinct}} \frac{1}{(x_k-x_{k'})^2 (x_{k''}-x_k)(x_{k''}-x_{k'})}.
\end{align*}
The left-hand side expands as 
\begin{align*} &\sum_{k,k' \in K:\, k \neq k'} \frac{4}{(x_k-x_{k'})^4} - \sum_{k,k' \in K:\, k \neq k'} \frac{4}{(x_k-x_{k'})^2} \sum_{k'' \in K:\, k'' \neq k,k'} \frac{1}{(x_{k''}-x_k)(x_{k''}-x_{k'})} \\
&\quad + \sum_{k,k' \in K:\, k \neq k'} \; \sum_{k'' \in K:\, k'' \neq k,k'} \frac{1}{(x_{k''}-x_k)^2 (x_{k''}-x_{k'})^2} \\
&\quad +  \sum_{k,k' \in K:\, k \neq k'} \; \sum_{k'',k''' \in K:\, k'' \neq k,k'} \frac{1}{(x_{k''}-x_k) (x_{k''}-x_{k'}) (x_{k'''}-x_k) (x_{k'''}-x_{k'})}. 
\end{align*}
The final sum can be rewritten as
$$ \sum_{k,k',k'',k''' \in K:\, k,k',k'',k'''\ \mathrm{distinct}} \frac{(x_k - x_{k'})(x_{k''}-x_{k'''})}{(x_{k''}-x_k) (x_{k''}-x_{k'}) (x_{k'''}-x_k) (x_{k'''}-x_{k'})(x_k-x_{k'})(x_{k''}-x_{k'''})}.$$
The denominator is a Vandermonde determinant and is totally antisymmetric in $k,k',k'',k'''$.  All the monomials appearing in the numerator disappear upon antisymmetrization, so the final sum vanishes.  To conclude the proof of (iii), it suffices to show that
\begin{multline*}
\sum_{k,k' \in K:\, k \neq k'} \; \sum_{k'' \in K:\, k'' \neq k,k'} \frac{1}{(x_{k''}-x_k)^2 (x_{k''}-x_{k'})^2} \\
=  \sum_{k,k' \in K: k \neq k'} \frac{2}{(x_k-x_{k'})^2} \sum_{k'' \in K:\, k'' \neq k,k'} \frac{1}{(x_{k''}-x_k)(x_{k''}-x_{k'})}.
\end{multline*}
The difference between the LHS and RHS can be written as
$$ \sum_{k,k',k'' \in K:\, k,k',k''\ \mathrm{distinct}} \frac{(x_k-x_{k'})^2 - 2 (x_{k''}-x_k) (x_{k''}-x_k)}{(x_{k''}-x_k)^2 (x_{k''}-x_{k'})^2 (x_k-x_{k'})^2}.$$
The denominator is totally symmetric in $k,k',k''$, while the numerator symmetrizes to zero, giving the claim.

Now we prove (iv).  The left-hand side expands as
$$ 2 \sum_{k,k' \in K:\, k \neq k'} (x_k-x_{k'}) \partial_t (x_k-x_{k'})$$
which by (i) becomes
$$ 8 |K| (|K|-1) - 4 \sum_{k,k' \in K:\, k \neq k'} (x_k-x_{k'})^2 \sum_{i \neq k,k'} \frac{1}{(x_i-x_k)(x_i-x_{k'})}.$$
It will thus suffice to show that
$$ \sum_{k,k',k'' \in K:\, k, k', k''\ \mathrm{distinct}} \frac{(x_k-x_{k'})^2}{(x_{k''}-x_k)(x_{k''}-x_{k'})} = -|K| (|K|-1) (|K|-2).$$
But the left-hand side can be written as
$$ \sum_{k,k',k'' \in K:\, k, k', k''\ \mathrm{distinct}} \frac{(x_k-x_{k'})^3}{(x_{k''}-x_k)(x_{k''}-x_{k'})(x_k - x_{k'})} = -|K| (|K|-1) (|K|-2).$$
The denominator is totally antisymmetric in $k,k',k''$.  The numerator antisymmetrizes to $-(x_{k''}-x_k)(x_{k''}-x_{k'})(x_k - x_{k'})$, giving the claim.

Finally we prove (v).  The left-hand side expands as
$$
-\sum_{k,k' \in K:\, k \neq k'} \frac{\partial_t (x_k-x_{k'})}{x_k-x_{k'}}$$
which by (i) becomes
$$
-\sum_{k,k' \in K:\, k \neq k'} \frac{4}{(x_k-x_{k'})^2} + 2 \sum_{k,k' \in K:\, k \neq k'} \sum_{i:\, i \neq k,k'} \frac{1}{(x_i-x_k)(x_i-x_{k'})}.$$
It thus suffices to show that the expression
$$ \sum_{k,k',k'' \in K: k,k',k''\ \mathrm{distinct}} \frac{1}{(x_{k''}-x_k)(x_{k''}-x_{k'})}$$
vanishes. But the summand antisymmetrizes to zero, giving the claim.
\end{proof}

\section{A weak bound on gaps}\label{weakgap}

In order to analyze (truncated versions) of the Hamiltonian ${\mathcal H}(t) = \sum_{j \neq k} H_{jk}(t)$, we will need some upper bounds on the individual terms $H_{jk}(t)$.   It was shown in \cite[Corollary 1]{csv} that these quantities are finite (i.e., the zeroes are simple) when $\Lambda < t \leq 0$.  It turns out that by refining the analysis in \cite{csv} (and by narrowing the range of times $t$ to the region $\Lambda/2 \leq t \leq 0$), one can establish a more quantitative lower bound:

\begin{proposition}[Lower bound on gaps]\label{gap}  For any $j \in \Z^*$ and any $\Lambda/2 \leq t \leq 0$, one has
\begin{equation}\label{mox}
 \max_{k \in \Z^*:\, k \neq j} H_{jk}(t) \ll (\log^2_+ j) \log_+ \log_+ j
\end{equation}
\end{proposition}

The bound in \eqref{mox} is probably not optimal, but for our application any bound that grows more slowly than (say) $|j|^{0.1}$ as $j \to \infty$ would suffice.

To prove this proposition, we first need the following variant of a result in \cite{csv}:

\begin{lemma}\label{ga}  Let $K$ be a finite subset of $\Z^*$ of cardinality $|K| \geq 2$, and let $\Lambda/2 \leq t \leq 0$.  Then
$$ \sum_{k,k' \in K:\, k \neq k'} (x_k(t)-x_{k'}(t))^2 \gg \frac{|K|^3}{1+\sum_{\substack{k \in K\\ j \not \in K}} E_{jk}(t)}.$$
\end{lemma}

Informally, this lemma asserts that the gaps within $K$ cannot be too small, unless there is also a small gap between an element of $K$ and an element outside of $K$.  The strategy will be to iterate this observation to show that a very small gap will therefore propagate until it contradicts \eqref{add}.

\begin{proof}  Let $A = A(t)$ and $B= B(t)$ denote the functions
\begin{align*}
A(t) &\coloneqq \sum_{k,k' \in K:\, k \neq k'} (x_k(t)-x_{k'}(t))^2 \\
B(t) &\coloneqq \sum_{\substack{k \in K \\ j \not \in K}} E_{kj}(t).
\end{align*}
The function $A(t)$ is continuously differentiable.  The corresponding claim for $B(t)$ is not obvious; however, 
the sum defining $B(t)$ is uniformly convergent (thanks to \eqref{xji0}) and hence $B(t)$ is at least continuous.
From Lemma \ref{ident}(ii) we have the lower bound
$$ \partial_{t'} B(t') \geq -8 B(t')^2$$
(cf. \cite[Lemma 2.5]{csv}) in the weak sense for $\Lambda < t' \leq 0$.  In particular, if there exists a time $\Lambda < t_- < t$ such that
$$ \sup_{t_- \leq t' \leq t} B(t') = B(t_-) = 2 B(t)$$
then we have
$$ B(t) - B(t_-) \geq -8 B(t)^2 (t-t_-)$$
which rearranges as
$$ t-t_- \geq \frac{1}{8B(t)}.$$
By continuity, we conclude that $B(t')$ cannot attain or exceed the value $2B(t)$ anywhere in the interval $(-\Lambda,t] \cap (t - \frac{1}{8B(t)}, t)$, that is to say that
$$ B'(t) < 2B(t)$$
whenever
$$ t - \frac{1}{8B(t)}, \Lambda < t' \leq t.$$
by hypothesis, this is a range of size at least
$$ \min(\frac{\Lambda}{2}, \frac{1}{16 B(t)}) \gg \frac{1}{1 + B(t)}.$$
On the other hand, for $t'$ in the above range, we see from Lemma \ref{ident}(iv) that
\begin{align*}
 \partial_{t'} A(t') &= 4|K|^2 (|K|-1) + O( B(t') A(t') ) \\
&= 4|K|^2 (|K|-1) + O( B(t) A(t') ) 
\end{align*}
and hence by Gronwall's inequality one has
$$
A(t) \gg \frac{4 |K|^2 (|K|-1)}{1+B(t)},$$
giving the claim.
\end{proof}

Now we fix a time $\Lambda/2 \leq t \leq 0$, and drop the dependence on $t$.
For any finite set $K \subset \Z^*$ with $|K| \geq 2$, set $\delta(K) := \max_{k,k' \in K} |x_k - x_{k'}|$ to be the largest gap in $K$.  Then
$$ \sum_{k,k' \in K:\, k \neq k'} (x_k-x_{k'})^2 \leq |K|^2 \delta(K)^2 $$
and so from the above lemma we have
$$ 1+\sum_{\substack{k \in K \\ j \not \in K}} E_{kj} \geq |K|^{-5} \delta(K)^{-2}.$$
In particular, if $\delta(K) \leq c |K|^{-5/2}$ for a sufficiently small absolute constant $c>0$, then we have
$$ \sum_{\substack{k \in K \\ j \not \in K}} E_{kj} \geq |K|^{-5} \delta(K)^{-2},$$
and hence by the pigeonhole principle there exists $k \in K$ such that
$$ \sum_{j \not \in K} E_{kj} \gg |K|^{-6} \delta(K)^{-2}.$$
From \eqref{add}, \eqref{ekj-def} we have
$$ \sum_{j \not \in K} E_{kj} \ll 1 + (\log^2_+ \xi_k) \min_{j \not \in K} |x_k-x_j|^{-2}.$$
We conclude that if $\delta(K) \leq c |K|^{-3}$ for a sufficiently small $c>0$, then there exists $k \in K$ such that
$$ (\log^2_+ \xi_k) \min_{j \not \in K} |x_k-x_j|^{-2} \gg |K|^{-6} \delta(K)^{-2}$$
or equivalently
$$ \min_{j \not \in K} |x_k-x_j| \ll |K|^3 \delta(K) \log_+ \xi_k. $$
Now suppose that $K$ is a discrete interval $[k_-,k_+]_{\Z^*}$ for some $1 < k_- < k_+$.  Then 
$$ \min_{j \not \in K} |x_k-x_j| \geq \min( |x_{k_-} - x_{k_--1}|, |x_{k_+} - x_{k_++1}| )$$
and thus (assuming that  $\delta(K) \leq c |K|^{-3}$) we have
$$ \min( |x_{k_-} - x_{k_--1}|, |x_{k_+} - x_{k_++1}| ) \ll |K|^3 \delta(K) \log_+ k_+$$
which implies that
\begin{equation}\label{kkp}
 \delta(K') \ll |K|^3 \log(k_+) \delta(K)
\end{equation}
whenever $\delta(K) \leq c |K|^{-3}$, where $K'$ is either the interval $K' = [k_- - 1, k_+]_{\Z^*}$ or $K' = [k_-, k_+ + 1]_{\Z^*}$.  In either case, we call $K'$ an \emph{enlargement} of $K$.

Now we can prove Proposition \ref{gap}.  By symmetry we may assume $j$ is positive.  We can also assume $j$ is large, as the claim follows from compactness for bounded $j$.  As before, we suppress the dependence on $t$.  It thus suffices to show that
$$ \log \frac{1}{|x_{j+1} - x_j|} \ll (\log^2 j) \log\log j $$
for large positive $j$.

By iterating \eqref{kkp} at most $\log j$ times starting from the interval $K_1 \coloneqq [j,j+1]_{\Z^*}$, we can find a sequence
$$ [j,j+1]_{\Z^*} = K_1 \subset K_2 \subset \dots \subset K_r$$
of discrete intervals $K_i = [k_{-,i}, k_{+,i}]_{\Z^*}$ for some $1 \leq r \leq \log^2_+ \xi_j$ with the following properties:
\begin{itemize}
\item[(i)] For each $1 \leq i < r$, $K_{i+1}$ an enlargement of $K_i$ with $\delta(K_{i+1}) \ll |K_i|^3 \delta(K_i) \log_+ k_{+,i}$.
\item[(ii)] Either $\delta(K_r) > c |K_r|^{-3}$, or $r+1 > \log^2_+ \xi_j$.
\end{itemize}
Since $|K_i| \leq r+1 \ll \log^2_+ \xi_j \ll \log^2 j$ and $k_{+,i} \leq j + r \ll j$, we have from property (i) that
$$ \delta(K_{i+1}) \ll j \log^2 j \delta(K_i)$$
for all $1 \leq i < r$, and hence
$$ \delta(K_r) \ll \exp( O( \log^2 j \log\log j ) ) \delta(K_1).$$
On the other hand, from property (ii), using the bound $|K_r| \leq r+1 \ll \log^2 \xi_j$ in the first case and \eqref{add} and the pigeonhole principle in the second case, we have
$$ \delta(K_r) \gg \log^{-6} \xi_j \gg \log^{-6} j.$$
Combining the two estimates, we obtain the claim.

\section{A weak bound on integrated energy}

In addition to truncations of the Hamiltonian, we will also need to control truncations of the energy $\sum_{j \neq k} E_{jk}(t)$.  While Proposition \ref{gap} provides some control on the summands here, it is too weak for our purposes (being of worse than polynomial growth in $j,k$), and we will need the following integrated bound that, while still weak, is at least of polynomial growth:

\begin{proposition}[Weak bound on integrated energy]\label{energy-weak}  Let $J > 0$.  Then
$$ \int_{\Lambda/2}^0 \sum_{J \leq j < k \leq 2J} E_{jk}(t)\, dt \ll J^2 \log_+^{O(1)} J.$$
\end{proposition}

We will use this bound to justify an interchange of a derivative and an infinite series summation in the next section.

\begin{proof}  We may take $J$ to be large, as the claim is trivial for $J$ in the compact region $J=O(1)$.  For any discrete interval $I$, let $Q_I$ denote the quantity
$$ Q_I \coloneqq \int_{\Lambda/2}^0 \sum_{j,k \in I: j \neq k} E_{jk}(t)\, dt.$$
From \eqref{add} we have a crude lower bound 
$$ Q_{[J,2J]_{\Z^*}} \gg J \log^{-O(1)} J$$
while from Proposition \ref{gap} we have an extremely crude upper bound
$$ Q_{[0.5 J, 3J]_{\Z^*}} \ll \exp( O( \log^2 J \log\log J ) ).$$
The ratio between $Q_{[0.5 J, 3J]_{\Z^*}}$ and $Q_{[J,2J]_{\Z^*}}$ is thus less than $(1+J^{-0.1})^{0.5 J / J^{0.1}}$.  By the pigeonhole principle, we can then therefore find an interval $K \coloneqq [J_-,J_+]_{\Z^*}$ containing $[J,2J]_{\Z^*}$ and contained in $[0.5 J + J^{0.1}, 3J - J^{0.1}]_{\Z^*}$, such that
\begin{equation}\label{pig}
 Q_{K'} \leq (1 + J^{-0.1}) Q_{K},
\end{equation}
where $K' \coloneqq [J_--J^{0.1}, J_+ + J^{0.1}]_{\Z^*}$ is a slight enlargement of $K$.  Next, we apply Lemma \ref{ident}(v) and use the fundamental theorem of calculus to obtain the identity
$$
\sum_{k,k' \in K:\, k \neq k'} H_{kk'}(\Lambda/2) - H_{kk'}(0) 
 = 4 Q_{K} - 2 \int_{\Lambda/2}^0 \sum_{\substack{j \not \in K \\ k,k' \in K:\, k \neq k'}} \frac{1}{(x_j(t)-x_k(t))(x_j(t)-x_{k'}(t))}\, dt.$$
From Proposition \ref{gap}, the left-hand side is $O( J^2 \log^{O(1)} J )$, thus
$$Q_{K}  \ll J^2 \log^{O(1)} J + 
\int_{\Lambda/2}^0 \sum_{\substack{j \not \in K \\ k,k' \in [J_-,J_+]_{\Z^*}:\, k \neq k'}} \frac{1}{(x_j(t)-x_k(t))(x_j(t)-x_{k'}(t))}\, dt.$$
Using $ab \ll a^2+b^2$, we thus have
$$Q_{K}  \ll J^2 \log^{O(1)} J + 
\int_{\Lambda/2}^0 \sum_{\substack{j \not \in K \\ k \in K}} \frac{1}{(x_j(t)-x_k(t))^2}\, dt.$$
Using \eqref{add}, the contribution to the integral of those $j$ outside of $K'$ may be crudely bounded by $O(J^2 \log^{O(1)} J)$ (in fact one can improve this bound to $O(J \log^{O(1)} J)$ if desired, although this will not help us significantly here).  The contribution of those $j$ inside $K'$ may be bounded by
$$ Q_{K'} - Q_{K} \leq J^{-0.1} Q_{K},$$
thanks to \eqref{pig}.  We conclude that
$$Q_{K}  \ll J^2 \log^{O(1)} J $$
and the claim follows.
\end{proof}

\section{Strong control on integrated energy}
\label{sec:strongcontrol}

As discussed previously, the strategy to establish convergence to local equilibrium is to study (a suitable variant of) the formal Hamiltonian
$$
 {\mathcal H}(t) = \sum_{j,k \in \Z^*: j \neq k} H_{jk}(t)
$$
and its derivatives, with the intention of controlling (suitable variants of) integrated energies such as
$$ \int_{\Lambda/4}^0 \sum_{j,k \in \Z^*: j \neq k} E_{jk}(t)\, dt.$$
Unfortunately, even with the bound just obtained in Proposition \ref{gap}, the above expression is far from being absolutely convergent.  To address this issue we need to mollify and renormalize the Hamiltonian and the energy in a number of ways.  We renormalize the inverse square function $x \mapsto \frac{1}{|x|^2}$ for $x \neq 0$ that appears in the definition of the energy interactions $E_{jk}(t)$ by introducing the modified potential
$$ V(x) \coloneqq \frac{1}{|x|^2} - 1 + 2(|x|-1),$$
which (for positive $x$) is $\frac{1}{x^2}$ minus the linearization $1 - 2(x-1)$ of that function at $x=1$.  As $\frac{1}{x^2}$ is convex, $V$ is non-negative, and one can verify the asymptotics
\begin{equation}\label{vlog}
\begin{split}
V(x) &\asymp \frac{1}{|x|^2} \quad \hbox{ for } |x| \leq 1/2 \\
V(x) &\asymp (|x|-1)^2 \quad \hbox{ for } 1/2 < |x| \leq 2 \\
V(x) &\asymp |x|  \quad \hbox{ for } |x| > 2.
\end{split}
\end{equation}
For any distinct $j,k$ and any $\Lambda/2 \leq t \leq 0$, we define the renormalization
$$ \tilde E_{jk}(t) \coloneqq \frac{1}{|\xi_k - \xi_j|^2} V\left( \frac{x_k(t) - x_j(t)}{\xi_k - \xi_j} \right)$$
of the interaction energy $E_{jk}(t)$;
we observe that
\begin{equation}\label{tejk}
 \tilde E_{jk}(t) = E_{jk}(t) - \frac{1}{|\xi_k-\xi_j|^2} + 2 \frac{(x_k(t)-\xi_k)-(x_j(t)-\xi_j)}{(\xi_k-\xi_j)^3}.
\end{equation}

For any discrete interval $I \subset \Z^*$, we define the renormalized energy
$$ \tilde E^I(t) \coloneqq \sum_{j,k \in I: j \neq k} \tilde E_{jk}(t);$$
this is clearly a non-negative quantity that is non-decreasing in $I$.  It can also be simplified up to negligible error as follows:

\begin{lemma}\label{eb}  If $I = [I_-,I_+]_{\Z^*}$ is a discrete interval and $\Lambda/2 \leq t \leq 0$, then
$$
\tilde E^I(t) = \left(\sum_{j,k \in I:\, j \neq k} E_{jk}(t) - \frac{1}{|\xi_k-\xi_j|^2}\right) + O( \log_+^{O(1)}( |I_-|+|I_+| ) ).$$
\end{lemma}

\begin{proof}  By symmetry and the triangle inequality we may assume without loss of generality that $0 \leq I_- \leq I_+$; we may then assume that $I_+$ is large, as the claim is trivial for $I_+$ in the compact region $I_+ = O(1)$.  By \eqref{tejk}, it suffices to show that
$$ \sum_{j,k \in I:\, j \neq k} \frac{(x_k(t)-\xi_k)-(x_j(t)-\xi_j)}{(\xi_k-\xi_j)^3} \ll \log^{O(1)} I_+.$$
We may desymmetrize the left-hand side as
$$ 2 \sum_{j \in I} (x_j(t)-\xi_j) \sum_{k \in I:\, k \neq j} \frac{1}{(\xi_k - \xi_j)^3}.$$
By \eqref{xji}, it thus suffices to show that
\begin{equation}\label{ji}
 \sum_{j \in I} \left| \sum_{k \in I:\, k \neq j} \frac{1}{(\xi_k - \xi_j)^3} \right| \ll \log^{O(1)} I_+.
\end{equation}
Consider the inner sum $\sum_{k \in I:\, k \neq j} \frac{1}{(\xi_k - \xi_j)^3}$.  From \eqref{xidif} we see that the contribution to this inner sum of those $k$ with $|k-j| \geq \frac{1}{2} j$ (say) is $O\left( \frac{\log^{O(1)} I_+}{j^2} \right)$.  For the remaining range $|k-j| < \frac{1}{2} j$, we can use \eqref{add-2} to estimate
$$
\frac{1}{(\xi_k - \xi_j)^3} = \frac{\log^3 \xi_j}{(4\pi)^3} \frac{1}{(k-j)^3} + O \left( \frac{\log^{O(1)} I_+}{j (k-j)^2} \right) $$
and so on summing we obtain
$$
\sum_{k \in I:\, k \neq j} \frac{1}{(\xi_k - \xi_j)^3} = \frac{\log^3 \xi_j}{(4\pi)^3} \sum_{k \in I:\, 0 < |k-j| < \frac{1}{2} j} \frac{1}{(k-j)^3}
+ O \left( \frac{\log^{O(1)} I_+}{j} \right).$$
As $k \mapsto \frac{1}{k-j}$ is odd around $j$, the sum on the right-hand side can be estimated as $O( \frac{1}{\max(|j-I_-|, |I_+-j|)^2} )$.  Using this bound we obtain \eqref{ji}.
\end{proof}

In this section we will establish the following significant improvement to Proposition \ref{energy-weak}:

\begin{theorem}\label{prelim}  For any $T>0$, one has
\begin{equation}\label{prelim-2}
\int_{\Lambda/4}^0 \tilde E^{[0.5 T \log T, 3 T \log T]_{\Z^\ast}}(t)\, dt = o_{T \to \infty}( T \log^3_+ T ).
\end{equation}
\end{theorem}

The remainder of this section is devoted to a proof of Theorem \ref{prelim}.
The claim is trivial for $T$ in any compact region $T=O(1)$, so we may assume without loss of generality that $T$ is large.  Recall the notation $X \lessapprox Y$ or $X = \tilde O(Y)$ for $X \ll Y \log^{O(1)} T$ introduced in the notation section of the paper; this will be convenient to use in the argument that follows. (Typically, when we use this notation, we will also have some sort of power gain $T^{-c}$ that will safely absorb all the $\log^{O(1)} T$ factors.)  Let $\psi_T: \Z^* \to \R^+$ be the weight function
\begin{equation}\label{psit-def}
 \psi_T(j) := \left(1 + \frac{|j|}{T \log T}\right)^{-100}.
\end{equation}
This is a smooth positive weight that is mostly localised to the region $j = O(T \log T)$ and fairly rapidly decaying away from this region.

We introduce the smoothly truncated renormalized energy
\begin{equation}\label{eform}
\tilde E_T(t) \coloneqq \sum_{j,k \in \Z^*:\, j \neq k} \psi_T(j) \psi_T(k) \tilde E_{jk}(t)
\end{equation}
for $\Lambda/2 \leq t \leq 0$.
This is clearly non-negative, and from Proposition \ref{energy-weak}, \eqref{psit-def}, \eqref{vlog}, and Fubini's theorem we see that $\tilde E_T$ is absolutely integrable in time (in particular, it is finite for almost every $\Lambda/2 \leq t \leq 0$).  Since the $E_{jk}$ are non-negative, we see from \eqref{psit-def} that to prove \eqref{prelim-2} it will suffice to show that
\begin{equation}\label{prelim-3}
\int_{\Lambda/4}^0 \tilde E_T(t)\, dt = o_{T \to \infty}( T \log^3_+ T ).
\end{equation}

We have an analogue of Lemma \ref{eb}:

\begin{lemma}\label{eb2}  For almost every $\Lambda/2 \leq t \leq 0$, one has
$$
\tilde E_T(t) = \left(\sum_{j,k \in \Z^*:\, j \neq k} \psi_T(j) \psi_T(k) \left(E_{jk}(t) - \frac{1}{|\xi_k-\xi_j|^2}\right)\right) + \tilde O(1).$$
\end{lemma}

\begin{proof}  For almost every $t$, one sees from Proposition \ref{energy-weak}, \eqref{psit-def}, and Fubini's theorem (and \eqref{xji0}) that the series
$$ \sum_{j,k:\, j \neq k} \psi_T(j) \psi_T(k) \left( E_{jk}(t)  + \frac{1}{|\xi_k-\xi_j|^2} + \frac{|x_j(t)| + |x_k(t)|}{|\xi_k-\xi_j|^3} \right)$$
is absolutely convergent.  Thus, by Fubini's theorem and \eqref{tejk}, it will suffice to show that
$$ \sum_{j,k \in \Z^*:\, j \neq k} \psi_T(j) \psi_T(k) \frac{(x_k(t)-\xi_k)-(x_j(t)-\xi_j)}{(\xi_k-\xi_j)^3} \lessapprox 1.$$
We may desymmetrize the left-hand side (again using Fubini's theorem) as
$$ 2 \sum_{j \in \Z^*} \psi_T(j) (x_j(t)-\xi_j) \sum_{k \in \Z^*:\, k \neq j} \frac{\psi_T(k)}{(\xi_k - \xi_j)^3},$$
and so it will suffice to establish the bound
$$
\sum_{k \in \Z^*:\, k \neq j} \frac{\psi_T(k)}{(\xi_k - \xi_j)^3} \lessapprox \frac{1}{|j|} + \frac{1}{T} $$
for all $j \in \Z^*$.

As in the proof of Lemma \ref{eb}, we see from \eqref{xidif} that the contribution of those $k$ with $|k-j| \geq \frac{1}{2} j$ is acceptable.  For the remaining range $|k-j| < \frac{1}{2} j$, we again use \eqref{add-2} to estimate
$$
\frac{1}{(\xi_k - \xi_j)^3} = \frac{\log^3 \xi_j}{(4\pi)^3} \frac{1}{(k-j)^3} + \tilde O \left( \frac{1}{j (k-j)^2} \right) $$
and similarly
$$ \psi_T(k) = \psi_T(j) + \tilde O \left( \frac{|k-j|}{T} \right),$$
and the claim follows by direct computation using the fact that $k \mapsto \frac{1}{k-j}$ is odd around $j$.
\end{proof}

Recall that two indices $j,k \in \Z^*$ are said to be \emph{nearby}, and we write $j \sim_T k$, if one has
$$ 0 < |j-k| < (T^2 + |j| + |k|)^{0.1}.$$
This is clearly a symmetric relation.  

Next, for $\Lambda/2 \leq t \leq 0$, we define the smoothly truncated renormalized Hamiltonian
\begin{equation}\label{hform}
\tilde {\mathcal H}_T(t) \coloneqq \sum_{j,k \in \Z^*:\, j \sim_T k} \psi_T(j) \psi_T(k) \left(H_{jk}(t) - \log \frac{1}{|\xi_j-\xi_k|}\right).
\end{equation}
From \eqref{psit-def}, Proposition \ref{gap}, and \eqref{add-2} we see that the sum here is absolutely convergent for every $\Lambda/2 \leq t \leq 0$.  We can also express it in terms of non-negative quantities plus a small error, in a manner similar to Lemma \ref{eb2}, as follows.  We first introduce the renormalization
$$ L(x) := \log \frac{1}{|x|} + |x| - 1$$
of the logarithm function $x \mapsto \log \frac{1}{|x|}$;
this is a convex nonnegative function on $\R \backslash \{0\}$ that vanishes precisely when $|x|=1$, and obeys the asymptotics
\begin{equation}\label{llog}
\begin{split}
L(x) &\asymp \log_+ \frac{1}{|x|} \quad \hbox{ for } 0 < |x| \leq 1/2 \\
L(x) &\asymp (|x|-1)^2 \quad \hbox{ for } 1/2 < |x| \leq 2 \\
L(x) &\asymp |x|  \quad \hbox{ for } |x| > 2.
\end{split}
\end{equation}
For any $\Lambda/2 \leq t \leq 0$ and distinct $j,k \in \Z^*$, we define the normalization
\begin{equation}\label{thkj-def}
 \tilde H_{jk}(t) \coloneqq L\left( \frac{x_j(t)-x_k(t)}{\xi_j-\xi_k} \right) 
\end{equation}
of the Hamiltonian interaction $H_{jk}(t)$; this is symmetric in $j,k$ and non-negative, vanishing precisely when $x_k(t)-x_j(t) = \xi_k-\xi_j$. 

\begin{lemma}\label{hamil-form}  For every $\Lambda/2 \leq t \leq 0$, one has
$$
\tilde {\mathcal H}_T(t) = \sum_{j,k \in \Z^*:\, j \sim_T k} \psi_T(j) \psi_T(k) \tilde H_{jk}(t) + o_{T \to \infty}(T \log^3_+ T).$$
\end{lemma}

\begin{proof} From \eqref{thkj-def} one has
$$ \tilde H_{jk}(t) = H_{jk}(t) - \log \frac{1}{|\xi_j-\xi_k|} - \frac{(x_j(t)-\xi_j) - (x_k(t)-\xi_k)}{\xi_j-\xi_k} $$
so by \eqref{hform} it suffices to show that
$$ \sum_{j,k \in \Z^*:\, j \sim_T k} \psi_T(j) \psi_T(k) \frac{(x_j(t)-\xi_j) - (x_k(t)-\xi_k)}{\xi_j-\xi_k} = o_{T \to \infty}(T \log^3_+ T).$$
Note from \eqref{xia}, \eqref{xji0}, \eqref{add-2}, \eqref{psit-def} that the sum here is absolutely convergent.  Desymmetrizing, it suffices to show that
$$ \sum_{j \in \Z^*} \psi_T(j) |x_j(t) - \xi_j| \left| \sum_{k \in \Z^*: j \sim_T k} \frac{\psi_T(k)}{\xi_j-\xi_k} \right| = o_{T \to \infty}(T \log^3 T).$$
The inner sum can be crudely bounded by $\tilde O(1)$ for all $j$ thanks to \eqref{psit-def}, \eqref{add-2}.  By \eqref{psit-def}, \eqref{xji}, \eqref{xia}, it thus suffices to show that
\begin{equation}\label{souse}
 \sum_{k \in \Z^*:\, j \sim_T k} \frac{\psi_T(k)}{\xi_j-\xi_k} = o_{T \to \infty}(\log T)
\end{equation}
whenever $T^{0.5} \leq |j| \leq T^{1.5}$ (say).  For $j \sim_T k$, one has $\psi_T(k) = \psi_T(j) + \tilde O( T^{-0.8} )$, and the contribution of the error term is acceptable by \eqref{add-2}, so it suffices to show that
\begin{equation}\label{xikjd}
 \sum_{k \in \Z^*:\, j \sim_T k} \frac{1}{\xi_j-\xi_k} = o_{T \to \infty}(\log T)
\end{equation}
whenever $|j| \geq T^{0.5}$. But from \eqref{add-2} we have
$$ \xi_j - \xi_k = \frac{4\pi}{\log \xi_j} (j-k) + O\left( \frac{|j-k|^2}{|j| \log^2_+ \xi_j} \right ) $$
and hence
\begin{equation}\label{stat}
 \frac{1}{\xi_j - \xi_k} = \frac{\log \xi_j}{4\pi} \frac{1}{j-k} + O\left( \frac{1}{|j|} \right ).
\end{equation}
As $k \mapsto \frac{\log \xi_j}{4\pi} \frac{1}{j-k}$ is odd around $j$, and the set $\{k: j \sim_T k\}$ is very nearly symmetric around $j$, it is then easy to establish
\eqref{xikjd} as required.
\end{proof}

In contrast to the non-normalized interaction $H_{jk}(t)$, the quantity $\tilde H_{jk}(t)$ is well controlled when $k$ and $j$ are far apart:

\begin{lemma}[Long-range decay of $\tilde H_{jk}$]\label{lrdec}  Let $j,k$ be distinct elements of $\Z^*$, and let $t$ be in the range $\Lambda/2 \leq t \leq 0$.
There exists a quantity $\eps(j)$ that goes to zero as $|j| \to \infty$, such that if $|k-j|  \geq \eps(j)^{-1} \log^2_+ \xi_j$, then
$$ \tilde H_{jk}(t) \ll \frac{\log^4_+(|j|+|k|) }{|k-j|^2},$$
and if $\eps(j) \log^2_+ \xi_j \leq |k-j| \leq \eps(j)^{-1} \log^2_+ \xi_j$, one has the refinement
$$ \tilde H_{jk}(t) \ll \eps(j)^2 \frac{\log^4_+ j }{|k-j|^2}.$$
Finally, in the remaining region $|k-j| < \eps(j) \log^2_+ \xi_j$, one has the crude bound
$$ \tilde H_{jk}(t) \ll (\log^2_+ j) \log_+ \log_+ j.$$
\end{lemma}

\begin{proof}  First suppose that $|k-j| \geq \frac{1}{2} |j|$ (so in particular $|k-j| \asymp |j|+|k|$).  From \eqref{xji} one has
$$ x_k(t) - x_j(t) = \xi_k - \xi_j + O( \log_+(|j|+|k|) )$$
while from \eqref{xidif} one has
$$ |\xi_k - \xi_j| \gg \frac{|j|+|k|}{\log_+(|j|+|k|)} $$
and thus
$$ \frac{x_k(t)-x_j(t)}{\xi_k-\xi_j} - 1 \ll \frac{\log_+^2(|j|+|k|)}{|j|+|k|},$$
and the claim then follows from \eqref{llog} (noting that the case $|j|+|k| = O(1)$ can be treated by compactness).

Now suppose that $\eps(j)^{-1} \log^2_+ \xi_j \leq |k-j| < \frac{1}{2} |j|$.  By symmetry we can take $j$ positive; we may also assume $j$ large, as the bounded case $j=O(1)$ may be treated by compactness.  From \eqref{xji} one then has
$$ x_k(t) - x_j(t) = \xi_k - \xi_j + O( \log j )$$
and from \eqref{xidif} and one has
\begin{equation}\label{xikj}
|\xi_k - \xi_j| \asymp \frac{|k-j|}{\log j} 
\end{equation}
and hence
$$ \frac{x_k(t)-x_j(t)}{\xi_k-\xi_j} - 1 \ll \frac{\log^2(j)}{|k-j|} \leq \eps(j).$$
The claim then follows from \eqref{llog}.

Next, suppose that $\eps(j) \log^2_+ \xi_j \leq |k-j| \leq \eps(j)^{-1} \log^2_+ \xi_j$.  In this case, from \eqref{add} (iterated $O(\eps(j)^{-1})$ times) and \eqref{add-2} we have
$$ x_k(t) - x_j(t) = \xi_k - \xi_j + o_{j \to \infty}( \eps(j)^{-1} \log j )$$
while from \eqref{xidif} we continue to have \eqref{xikj}, and hence
$$ \frac{x_k(t)-x_j(t)}{\xi_k-\xi_j} - 1 = o_{j \to \infty}\left( \eps(j)^{-1} \frac{\log^2(j)}{|k-j|} \right),$$
with the decay rate in the $o_{j \to \infty}$ notation independent of the choice of function $\eps()$.  For $\eps(j)$ going to zero sufficiently slowly, the claim once again follows from \eqref{llog}.

Finally, for the remaining case $|k-j| < \eps(j) \log^2_+ \xi_j$ (which implies $x_k(t) - x_j(t) \ll \log^2_+ j$ thanks to \eqref{add}) the claim follows from Proposition \ref{gap} and \eqref{llog}.
\end{proof}

We call a (time-dependent) quantity \emph{moderately sized} if it is of the form $O( T \log^3_+ T + \tilde E_T(t) )$, and \emph{negligible} if it is of the form $o_{T \to \infty}(  T\log^3_+ T + \tilde E_T(t) )$.  The following lemma gives some examples of moderately sized and negligible quantities:

\begin{lemma}\label{neg}  Let $t$ be in the range $\Lambda/2 \leq t \leq 0$.
\begin{itemize}
\item[(i)]  The quantity
$$ \sum_{j,k \in \Z^*:\, j \neq k} \frac{\psi_T(j) \psi_T(k) }{|x_j(t)-x_k(t)|^2}$$
is moderately sized.
\item[(ii)]  The quantity
$$ (\log_+ T) \sum_{j,k \in \Z^*:\, j \neq k} \frac{\psi_T(j) \psi_T(k) }{|x_j(t)-x_k(t)|}$$
is moderately sized.
\item[(iii)]  For any absolute constants $C, c>0$, the expression
$$ (\log^C_+ T) \sum_{j,k \in \Z^*:\, |j|, |k| \leq T^{1-c}} \frac{\psi_T(j) \psi_T(k) }{|x_j(t)-x_k(t)|}$$
is negligible.
\item[(iv)]  For any absolute constants $C, c>0$, the expression
$$ (\log^C_+ T) \sum_{j,k \in \Z^*:\, |j|, |k| \geq T^{1+c}} \frac{\psi_T(j) \psi_T(k)}{|x_j(t)-x_k(t)|}$$
is negligible.
\end{itemize}
Similarly if the $x_i(t)$ are replaced by $\xi_i$ throughout.
\end{lemma}

\begin{proof}  For brevity we omit the explicit dependence on the time $t$.  Also, all summation indices $i,j,k$ are understood to range in $\Z^*$.

From \eqref{xidif} we see that
$$ \sum_{k:\, k \neq j} \frac{1}{|\xi_j-\xi_k|^2} \ll \log^2_+ j $$
for all $j \in \Z^*$, and hence
$$ \sum_{j,k:\, j \neq k} \psi_T(j) \psi_T(k) \frac{1}{|\xi_j-\xi_k|^2} \ll T \log^3_+ T.$$
From this and Lemma \ref{eb} we conclude (i).  Using
$$ \frac{\log_+ T}{|x_j(t)-x_k(t)|} \leq \frac{1}{|x_j(t)-x_k(t)|^2} + \log_+^2 T$$
we then obtain (ii).  If instead we use
$$ \frac{\log^C_+ T}{|x_j(t)-x_k(t)|} \leq \frac{1}{\log_+ T} \frac{1}{|x_j(t)-x_k(t)|^2} + \log_+^{2C+1} T$$
we obtain (iii) and (iv).  Similarly if the $x_i$ are replaced by $\xi_i$ throughout.
\end{proof}

We now have the following crucial derivative computation:

\begin{proposition}\label{dorium}  In the range $\Lambda/2 \leq t \leq 0$, the function ${\mathcal H}_T$ is absolutely continuous, and the derivative $\partial_t \tilde {\mathcal H}_T(t)$ is equal to $-4 \tilde E_T(t)$ plus negligible terms for almost all $t$.  In other words, one has
\begin{equation}\label{tash}
\partial_t \tilde{\mathcal H}_T(t) = -4 \tilde E_T(t) + o_{T \to \infty}\left( T \log^3 T + \tilde E_T(t) \right)
\end{equation}
for almost every $t$.
\end{proposition}

\begin{remark} This may be compared with Lemma \ref{ident}(v) or indeed the formal identity \eqref{formal_e}. That the right hand side is approximated in terms the renormalized energy, rather than just the energy, may be thought of heuristically as being a result of $\partial_t \mathcal{H}$ vanishing when the zeros $x_j$ settle on an equilibrium, being spaced like the points $\xi_j$.
\end{remark}

\begin{proof}  As before, we omit the explicit dependence on $t$, and all summation indices are understood to lie in $\Z^*$.
By \eqref{ode} we have
\begin{equation}\label{oda}
\partial_t H_{jk}(t) = -\frac{2}{x_k-x_j}
\left(\sum_{i:\, i \neq k}^\prime \frac{1}{x_k - x_i} - \sum_{i:\, i \neq j}^\prime \frac{1}{x_j - x_i}\right).
\end{equation}
If we \emph{formally} insert this into \eqref{hform}, and desymmetrize in $j$ and $k$, we would obtain the identity
\begin{equation}\label{formal-oda} 
\partial_t \tilde {\mathcal H}_T = 
-4\sum_{j,k:\, j \sim_T k} \psi_T(j) \psi_T(k) \frac{1}{x_k-x_j} \sum_{i:\, i \neq k}^\prime \frac{1}{x_k - x_i}.
\end{equation}
However, we need to justify the interchange of the derivative and the infinite summation.  First, we use the fundamental theorem of calculus to rewrite \eqref{oda} in integral form as
$$ H_{jk}(0) - H_{jk}(t_0) = -2 \int_{t_0}^0 \frac{1}{x_k(t)-x_j(t)} 
\left(\sum_{i:\, i \neq k}^\prime \frac{1}{x_k(t) - x_i(t)}- \sum_{i:\, i \neq j}^\prime \frac{1}{x_j(t) - x_i(t)}\right)\, dt$$
for any $\Lambda/2 \leq t_0 \leq 0$.  Multiplying by $\psi_T(j) \psi_T(k)$, we conclude that
\begin{multline*} {\mathcal H}_T(0) - {\mathcal H}_T(t_0) 
= - 2 \sum_{j,k: j \sim_T k} \psi_T(j) \psi_T(k) \int_{t_0}^0 \frac{1}{x_k(t)-x_j(t)} \\
\times
\left( \sum_{i:\, i \neq k}^\prime \frac{1}{x_k(t) - x_i(t)} - \sum_{i:\, i \neq j}^\prime \frac{1}{x_j(t) - x_i(t)} \right)\, dt.
\end{multline*}
By the dominated convergence theorem, we can interchange the outer sum and the integral as soon as we can show that the expression
$$
 \sum_{j,k: j \sim_T k} \psi_T(j) \psi_T(k) \int_{t_0}^0 \frac{1}{|x_k(t)-x_j(t)|}
\left( \left|\sum_{i:\, i \neq k}^\prime \frac{1}{x_k(t) - x_i(t)}\right| + \left|\sum_{i:\, i \neq j}^\prime \frac{1}{x_j(t) - x_i(t)}\right| \right)\, dt 
$$
is finite.  By symmetry in $j$ and $k$, it suffices to show that
\begin{equation}\label{Stil}
 \sum_{j,k: j \sim_T k} \psi_T(j) \psi_T(k) \int_{t_0}^0 \frac{1}{|x_k(t)-x_j(t)|}
\left|\sum_{i:\, i \neq k}^\prime \frac{1}{x_k(t) - x_i(t)}\right|\, dt 
\end{equation}
is finite.  But using \eqref{add}, \eqref{xji} we can crudely bound
$$ \left|\sum_{i:\, i \neq k}^\prime \frac{1}{x_k(t) - x_i(t)}\right|, \frac{1}{|x_k(t)-x_j(t)|}  \ll \log_+^{O(1)}(k) \left( \frac{1}{|x_k(t)-x_{k-1}(t)|} + \frac{1}{|x_k(t)-x_{k+1}(t)|} \right)$$
(using the convention $x_0(t)=0$), so the expression \eqref{Stil} may in turn be crudely bounded by
$$ \sum_k \psi_T^2(k) (T + |k|)^{0.1} \log_+^{O(1)}(k) \int_{t_0}^0 \frac{1}{|x_k(t)-x_{k-1}(t)|^2} + \frac{1}{|x_k(t)-x_{k+1}(t)|^2}\, dt,$$
and this will be finite thanks to Proposition \ref{energy-weak} and \eqref{psit-def}.  We conclude (after desymmetrizing in $j$ and $k$) that
$$ {\mathcal H}_T(0) - {\mathcal H}_T(t_0) = - 4 \int_{t_0}^0 \sum_{j,k: j \sim_T k} \psi_T(j) \psi_T(k) \frac{1}{x_k(t)-x_j(t)}
\sum_{i:\, i \neq k}^\prime \frac{1}{x_k(t) - x_i(t)}\, dt.$$
The above analysis also shows that the integrand is absolutely integrable in time.  From the Lebesgue differentiation theorem, we conclude that $\tilde {\mathcal H}_T$ is absolutely continuous and that \eqref{formal-oda} holds at almost every time $t$.

To conclude the proof of the proposition, it will thus suffice to show that
\begin{equation}\label{sumjk}
 \sum_{j,k: j \sim_T k} \psi_T(j) \psi_T(k) \frac{1}{x_k-x_j} 
\sum_{i \neq k}^\prime \frac{1}{x_k - x_i}
\end{equation}
is equal to $\tilde E_T$ plus negligible terms.  We can split this expression as $X_1 + X_2 + X_3 + X_4$, where
\begin{align*}
X_1 &\coloneqq \sum_{j,k: j \sim_T k} \psi_T(j) \psi_T(k) \frac{1}{(x_k-x_j)^2}  \\
X_2 &\coloneqq \sum_{j,k: j \sim_T k} \psi_T(j) \psi_T(k) \frac{1}{x_k-x_j} 
\sum_{i:\, i \sim_T j, k} \frac{1}{x_k - x_i} \\
X_3 &\coloneqq \sum_{j,k: j \sim_T k} \psi_T(j) \psi_T(k) \frac{1}{x_k-x_j} 
\sum_{i:\, i \sim_T k;\, i \not \sim_T j;\, i\neq j} \frac{1}{x_k - x_i} \\
X_4 &\coloneqq \sum_{j,k: j \sim_T k} \psi_T(j) \psi_T(k) \frac{1}{x_k-x_j} 
\sum_{i:\, i \not \sim_T k;\,  i \neq k}^\prime \frac{1}{x_k - x_i}.
\end{align*}

We first claim that $X_4$ is negligible.  From \eqref{xji} we have
$$ x_k - x_i = \xi_k - \xi_i + O( \log_+(|i|+|k|) )$$
and hence (by \eqref{xidif})
$$ \frac{1}{x_k - x_i} = \frac{1}{\xi_k - \xi_i} + O\left( \frac{\log^2_+(|i|+|k|)}{|k-i|^2} \right),$$
which implies that
$$ 
\sum_{i:\, i \not \sim_T k;\,  i \neq k}^\prime \frac{1}{x_k - x_i} = \sum_{i:\, i \not \sim_T k;\, i \neq k}^\prime \frac{1}{\xi_k - \xi_i}
+ \tilde O( T^{-0.1} ).$$
From \eqref{xidif} we may crudely bound this sum by $\tilde O(1)$.  By Lemma \ref{neg}(iii), this shows that the contribution to $X_4$ of those $k$ for which $|k| \leq T^{0.9}$ or $|k| \geq T^{1.1}$ (say) is negligible, so we may assume $T^{0.9} \leq |k| \leq T^{1.1}$.  Let $A \geq 2$ be a large constant.  Using \eqref{xidif} we may write
$$
\sum_{i:\, i \not \sim_T k;\, i \neq k}^\prime \frac{1}{\xi_k - \xi_i} = \sum_{i:\, T^{0.2} \leq |k-i| \leq A |k|} \frac{1}{\xi_k - \xi_i}
+ \sum_{i:\, |i| \geq A |k|} \frac{1}{\xi_k - \xi_i} + O\left( \frac{\log T}{A}  \right).$$
For the first sum on the right-hand side, we use \eqref{add-2} (as in the proof of \eqref{stat}) as well as \eqref{xia} to conclude that
$$ \frac{1}{\xi_k - \xi_i} = \frac{\log \xi_k}{4\pi} \frac{1}{k-i} + O_A\left( \frac{1}{|k|} \right ),$$
where the subscript in the $O_A$ notation means that the implied constant can depend on $A$.  As $i \mapsto \frac{\log \xi_k}{4\pi} \frac{1}{k-i} $ is odd around $k$, we conclude that
$$ \sum_{i:\, T^{0.2} \leq |k-i| \leq A |k|} \frac{1}{\xi_k - \xi_i} = O_A( 1 ).$$
Meanwhile, combining the $i$ and $-i$ terms and using \eqref{xidif}, \eqref{xia} we have
$$ \sum_{i:\, |i| \geq A |k|} \frac{1}{\xi_k - \xi_i} = - 2 \xi_k \sum_{i:\, i \geq A|k|} \frac{1}{\xi_i^2 - \xi_k^2} = O\left( \frac{\log T}{A} \right).$$
Sending $A$ slowly to infinity, we conclude that
$$ 
\sum_{i:\,  i \not \sim_T k;\, i \neq k}^\prime \frac{1}{x_k - x_i}  = o_{T \to \infty}( \log T )$$
and the negligibility of $X_4$ then follows from Lemma \ref{neg}(ii).

Now we claim that $X_2$ is negligible.  Thanks to the restrictions on $i,j,k$, we see that
$$ \psi_T(i), \psi_T(j) = \left(1 + \tilde O\left((T+|k|)^{-0.8}\right)\right) \psi_T(k)$$
and hence
$$ \psi_T(j) \psi_T(k) = \psi_T(i)^{2/3} \psi_T(j)^{2/3} \psi_T(k)^{2/3} + \tilde O( (T+|k|)^{-0.8} \psi_T(j) \psi_T(k)).$$
The sum
$$ \sum_{i,j,k:\, j \sim_T k;\, i \sim_T j, k} \frac{\psi_T(i)^{2/3} \psi_T(j)^{2/3} \psi_T(k)^{2/3}}{(x_k-x_j)(x_k - x_i)} $$
symmetrises to zero, and hence
$$ X_2 \lessapprox \sum_{i,j,k:\, j \sim_T k;\, i \sim_T j, k} (T+|k|)^{-0.8} \frac{\psi_T(j) \psi_T(k)}{|x_k-x_j| |x_k - x_i|}.$$
Estimating $\frac{1}{|x_k-x_j| |x_k - x_i|} \ll \frac{1}{|x_k-x_j|^2} + \frac{1}{|x_k - x_i|^2}$ and performing the $i$ or $j$ summation respectively, we conclude that
$$ X_2 \lessapprox \sum_{j,k:\, j \sim_T k} (T+|k|)^{-0.6} \frac{\psi_T(j) \psi_T(k)}{|x_k-x_j|^2}$$
and so $X_2$ is negligible thanks to Lemma \ref{neg}(i).

We have shown that the expression \eqref{sumjk} is equal to $X_1+X_3$ plus negligible terms.  A similar argument (replacing $x_i$ with $\xi_i$ throughout) shows that the expression
\begin{equation}\label{sumjk-2}
 \sum_{j,k:\, j \sim_T k} \psi_T(j) \psi_T(k) \frac{1}{\xi_k-\xi_j} \sum_{i:\, i \neq k}^\prime \frac{1}{\xi_k - \xi_i}
\end{equation}
is equal to $X'_1+X'_3$ plus negligible terms, where
\begin{align*}
X'_1 &\coloneqq \sum_{j,k:\, j \sim_T k} \psi_T(j) \psi_T(k) \frac{1}{(\xi_k-\xi_j)^2}  \\
X'_3 &\coloneqq \sum_{j,k:\, j \sim_T k} \psi_T(j) \psi_T(k) \frac{1}{\xi_k-\xi_j} 
\sum_{i:\, i \sim_T k;\, i \not \sim_T j;\, i \neq j} \frac{1}{\xi_k - \xi_i}.
\end{align*}
From Lemma \ref{eb2}, we see that $\tilde E_T$ is equal to 
\begin{equation}\label{mang}
X_1 - X'_1 + \sum_{j,k: j \not \sim_T k;\, j \neq k}\psi_T(j) \psi_T(k) \left(\frac{1}{(x_k-x_j)^2} - \frac{1}{(\xi_k-\xi_j)^2}\right)
\end{equation}
up to negligible terms.  From \eqref{xia}, \eqref{xidif} we have
$$ \frac{1}{(x_k-x_j)^2} - \frac{1}{(\xi_k-\xi_j)^2} \lessapprox \frac{\log_+^{O(1)}(|j|+|k|)}{|k-j|^3} $$ 
when $j \neq k$ and $j \not \sim_T k$, so the final term in \eqref{mang} is negligible.  Thus, to complete the proof of the proposition, it will suffice to show that the expression \eqref{sumjk-2} and the difference $X_3 - X'_3$ are both negligible.

The expression \eqref{sumjk-2} may be rearranged as
$$ \sum_k \psi_T(k) \left(\sum_{j:\, j \sim_T k} \frac{\psi_T(j)}{\xi_k-\xi_j}\right) \left(\sum_{i:\, i \neq k}^\prime \frac{1}{\xi_k - \xi_i}\right).$$
By \eqref{xidif}, both inner sums are $\tilde O(1)$, so the contribution of those $|k| \leq T^{0.5}$ or $|k| \geq T^{1.5}$ (say) are negligible.  For $T^{0.5} < |k| \leq T^{1.5}$, we see from \eqref{souse} that the factor $\sum_{j: j \sim_T k} \frac{\psi_T(j)}{\xi_k-\xi_j}$ is $o_{T \to \infty}( \log T )$, and from \eqref{xikjd}, \eqref{xidif}, and the triangle inequality we also see that $\sum_{i:\, i \neq k}^\prime \frac{1}{\xi_k - \xi_i} = O( \log T )$.  Thus \eqref{sumjk-2} is negligible as required.

Finally, we show that $X_3 - X'_3$ is negligible.  This quantity may be written as
$$
\sum_{i,j,k:\, i, j \sim_T k;\, |i-j| > (T^2+|i|+|j|)^{0.1}} 
\psi_T(j) \psi_T(k) ( \frac{1}{(x_k-x_j)(x_k - x_i)} - \frac{1}{(\xi_k-\xi_j)(\xi_k - \xi_i)} ).$$
Observe that if $|k-j|$ and $|k-i|$ are both larger than or equal to $T^{0.1}$, then from \eqref{xji}, \eqref{xidif} one has
$$ \frac{1}{(x_k-x_j)(x_k - x_i)} - \frac{1}{(\xi_k-\xi_j)(\xi_k - \xi_i)}  \ll \frac{\log^{O(1)}_+(|i|+|j|+|k|)}{ T^{0.1} |\xi_k-\xi_j| |\xi_k - \xi_i|} \ll \frac{\log^{O(1)}_+(|i|+|j|+|k|)  }{T^{0.1} |k-j| |k-i|},$$
and so the contribution of this case is negligible.  From the triangle inequality, we see that it is not possible for 
$|k-j|$ and $|k-i|$ to both be less than $T^{0.1}$, so it remains to treat the components
\begin{equation}\label{con1}
\sum_{\substack{i,j,k:\, 0 < |j-k| < (T^2 + |j| + |k|)^{0.1} \\ 0 < |i-k| < T^{0.1};\, |i-j| > (T^2+|i|+|j|)^{0.1}}} 
\psi_T(j) \psi_T(k) \left( \frac{1}{(x_k-x_j)(x_k - x_i)} - \frac{1}{(\xi_k-\xi_j)(\xi_k - \xi_i)} \right)
\end{equation}
and
\begin{equation}\label{con2}
\sum_{\substack{i,j,k:\, 0 < |j-k| < T^{0.1} \\ 0 < |i-k| < (T^2+|i|+|k|)^{0.1};\,  |i-j| > (T^2+|i|+|j|)^{0.1}}}
\psi_T(j) \psi_T(k) \left( \frac{1}{(x_k-x_j)(x_k - x_i)} - \frac{1}{(\xi_k-\xi_j)(\xi_k - \xi_i)} \right).
\end{equation}

Consider first \eqref{con1}.  From the triangle inequality we have $|j-k| \gg T^{0.2}$, and hence by \eqref{xji}
$$\frac{1}{x_k -x_j} = (1 + \tilde O(T^{-0.2})) \frac{1}{\xi_k - \xi_j}.$$
By Lemma \ref{neg}(ii) and \eqref{xidif} we may thus replace $\frac{1}{x_k-x_j}$ by $\frac{1}{\xi_k-\xi_j}$ at negligible cost in \eqref{con1}, leaving us with
$$
\sum_{\substack{i,j,k:\, 0 < |j-k| < (T^2 + |j| + |k|)^{0.1} \\ 0 < |i-k| < T^{0.1};\, |i-j| > (T^2+|i|+|j|)^{0.1}}} 
\psi_T(j) \psi_T(k) \left( \frac{1}{x_k-x_i} - \frac{1}{\xi_k-\xi_i} \right) \frac{1}{\xi_k-\xi_j}
$$
up to negligible errors.  But by \eqref{add-2} and the hypothesis $|i-k| \leq T^{0.1}$, one may bound
$$ \sum_{\substack{j:\, 0 < |j-k| < (T^2 + |j| + |k|)^{0.1} \\ |i-j| > (T^2+|i|+|j|)^{0.1}}} \frac{\psi_T(j)}{|\xi_k - \xi_j|} \lessapprox T^{-0.1} \psi_T(k)$$
when $T^{0.9} \leq |k| \leq T^{1.1}$, and use the weaker bound
$$ \sum_{\substack{j:\, 0 < |j-k| < (T^2 + |j| + |k|)^{0.1} \\ |i-j| > (T^2+|i|+|j|)^{0.1}}} \frac{\psi_T(j)}{|\xi_k - \xi_j|} \lessapprox \psi_T(k)$$
for all other $k$, so this expression is also negligible by Lemma \ref{neg}(ii), (iii), (iv) (noting that $\psi_T(k)$ and $\psi_T(i)$ are comparable).  A similar argument also handles \eqref{con2}.
\end{proof}

To use Proposition \ref{dorium}, we need estimates that ensure $\tilde E_T$ is large when $\tilde {\mathcal H}_T$ is large.  To this end we have

\begin{lemma}\label{lame}  Let $m$ be a natural number, and let $\Lambda/2 \leq t \leq 0$.  Let $T > 0$, and let $\delta = \delta(T)$ go to zero as $T \to \infty$ sufficiently slowly. If $\tilde {\mathcal H}_T(t) \geq \delta m T \log^3_+ T$, then $\tilde E_T(t) \gg \delta 2^{2m} T \log^3_+ T$ where the implied constant is absolute.
\end{lemma}

\begin{proof}  As before, we suppress explicit dependence on $t$, and we may assume $T$ to be large as the claim is trivial from compactness for $T=O(1)$.  From Lemma \ref{hamil-form} we have (for $\delta$ decaying sufficiently slowly) that
$$ \sum_{j,k \in \Z^*:\, j \sim_T k} \psi_T(j) \psi_T(k) \tilde H_{jk}(t) \geq \frac{99}{100} \delta m T \log^3 T.$$
From Lemma \ref{lrdec} we see that
$$ \sum_{k:\, j \sim_T k;\, |k-j| \geq \eps(j) \log^2_+ \xi_j} \tilde H_{jk}(t) \ll \eps(j) \log^2_+ j$$
for any $j \in \Z^*$, which implies that
$$ \sum_{j,k:\, j \sim_T k;\, |k-j| \geq \eps(j) \log^2_+ \xi_j} \psi_T(j) \psi_T(k) \tilde H_{jk}(t) \leq \frac{1}{2} \delta T \log^3 T$$
if $\delta(T)$ goes to zero slowly enough.  By \eqref{hform}, we conclude that
\begin{equation}\label{opt0}
\sum_{j,k:\, j \sim_T k;\, |k-j| < \eps(j) \log^2_+ \xi_j} \psi_T(j) \psi_T(k) \tilde H_{jk}(t) \gg \delta m T \log^3_+ T
\end{equation}

We now claim that
\begin{equation}\label{opt1}
\sum_{\substack{j,k:\, j \sim_T k; \, |k-j| < \eps(j) \log^2_+ \xi_j \\ |x_j-x_k| \geq 2^{-m} |\xi_j-\xi_k|}} \psi_T(j) \psi_T(k) \tilde H_{jk}(t) \ll \delta^2 m T \log^3 T
\end{equation}
(say).  To see this, we use \eqref{llog} and \eqref{add-2} to bound
$$ L_{jk} \ll m + \frac{|x_j-x_k|}{|\xi_j-\xi_k|} \ll m + \frac{|x_j-x_k|}{|j-k|} \log T$$
and also $\psi_T(j) \asymp \psi_T(k)$ for $j,k$ in the sum.  Thus we may bound \eqref{opt1} by
$$
m \sum_{j,k:\, |k-j| < \eps(j) \log^2_+ \xi_j} \psi_T(j)^2 + 
\sum_{j,k:\, 0 < |k-j| < \eps(j) \log^2_+ \xi_j} \psi_T(j)^2 \frac{|x_j-x_k|}{|j-k|} \log T.$$
We may directly compute
$$ \sum_{j,k:\, j \sim_T k;\,  |k-j| < \eps(j) \log^2_+ \xi_j} \psi_T(j)^2  \ll \delta^2 T \log^3 T$$
if $\delta = \delta(T)$ goes to zero slowly enough.  Thus it will suffice to show that
\begin{equation}\label{clam}
 \sum_{j,k:\, 0 < |k-j| < \eps(j) \log^2_+ \xi_j} \psi_T(j)^2 \frac{|x_j-x_k|}{|j-k|} \ll \delta^2 T \log^2 T.
\end{equation}
But for any natural number $n$, we see from telescoping series and \eqref{xji} that
$$ \sum_{j:\, 2^n \leq |j| < 2^{n+1}} |x_j - x_{j+h}| \ll |h| \frac{2^n}{n} $$
whenever $|h| \ll 2^n$; summing over $|h| < \eps(j) \log^2_+\xi_j$, we conclude that
$$ \sum_{\substack{j,k:\, 2^n \leq |j| < 2^{n+1} \\  0 < |k-j| < \eps(j) \log^2_+ \xi_j}} \frac{|x_j-x_k|}{|j-k|} \ll \eps(2^n) 2^n n $$
which gives \eqref{clam} if $\delta$ goes to zero slowly enough.

From \eqref{opt0} and \eqref{opt1} we have
$$\sum_{\substack{j,k: j \sim_T k;\, |k-j| < \eps(j) \log^2_+ \xi_j \\ |x_j-x_k| \leq 2^{-m} |\xi_j-\xi_k|}} \psi_T(j) \psi_T(k) \tilde H_{jk}(t) \gg \delta m T \log^3 T.$$
But for $j,k$ in this sum, we see from \eqref{llog}, \eqref{vlog} that
$$ \tilde H_{jk}(t)  \ll \log \frac{|\xi_j-\xi_k|}{|x_j-x_k|} \ll \frac{m 2^{-2m} |\xi_j-\xi_k|^2}{|x_j-x_k|^2} \ll m 2^{-2m} \tilde E_{jk}
$$
and the claim follows.
\end{proof}

We can now shrink $\tilde {\mathcal H}_T$ down to a reasonable size in finite time:

\begin{corollary}\label{core}  One has $\tilde {\mathcal H}_T(t) = O( \delta T \log^3_+ T )$ for $\Lambda/4 \leq t \leq 0$.
\end{corollary}

\begin{proof}  We may take $T$ to be large.  From Proposition \ref{dorium} and Lemma \ref{lame}, we see that for any natural number $m$, and for almost every time $t$ for which one has
$$ \tilde {\mathcal H}_T(t) \geq \delta m T \log^3 T,$$
one has
$$ \partial_t \tilde {\mathcal H}_T(t) \leq - c \delta 2^{2m} T \log^3 T$$
for some absolute constant $c>0$.  In particular, if $m$ is larger than some large absolute constant $m_0$, and $\Lambda/2 \leq t \leq \Lambda/4$ is such that 
\begin{equation}\label{dem}
 \delta m T \log^3 T \leq \tilde {\mathcal H}_T(t) \leq \delta (m+1) T \log^3 T,
\end{equation}
then it is not possible (for $m_0$ large enough) to have $\tilde {\mathcal H}_T(t') \geq \delta m T \log^3 T$ for all $t \leq t' \leq t + c^{-1} 2^{-2m}$, as this would violate the fundamental theorem of calculus for absolutely continuous functions.  Thus, by the intermediate value theorem, there exists $t \leq t' \leq t + c^{-1} 2^{-2m}$ such that
$$ \delta (m-1) T \log^3 T \leq \tilde {\mathcal H}_T(t') \leq \delta m T \log^3 T,$$
and on iterating this we conclude (for $m_0$ large enough) that there exists $t \leq t'' \leq t + 2 c^{-1} 2^{-2m_0}$ such that
\begin{equation}\label{ham}
 \tilde {\mathcal H}_T(t'') \leq \delta m_0 T \log^3 T.
\end{equation}
We run this argument with $t$ set equal to $\Lambda/2$, and $m$ the unique integer obeying \eqref{dem}, to conclude (for $m_0$ large enough) that there exists $\Lambda/2 \leq t'' \leq \Lambda/4$ obeying \eqref{ham}.  (Note that this conclusion is immediate if the initial value of $m$ was already less than $m_0$.)  On the other hand, from Proposition \ref{dorium} we have $\partial_t \tilde {\mathcal H}_T(t) \leq O(\delta T \log^3 T)$ for almost every $t'' \leq t \leq 0$, if $\delta$ decays sufficiently slowly.  The claim now follows from the fundamental theorem of calculus (absorbing $m_0$ into the implied constants), recalling that $\tilde {\mathcal H}_T$ is non-negative.
\end{proof}

From Proposition \ref{dorium} and the fundamental theorem of calculus for absolutely continuous functions, one has
$$ \tilde {\mathcal H}_T(\Lambda/4) - \tilde {\mathcal H}_T(0) = (4 + o_{T \to \infty}(1)) \int_{\Lambda/4}^0 \tilde E_T(t)\, dt + o_{T \to \infty}( T \log^3_+ T )$$
and the claim \eqref{prelim-3} now follows from Corollary \ref{core}.  This concludes the proof of Theorem \ref{prelim}.

\section{Controlling the energy at time $0$}

In the previous section we controlled a time average of the energy.  Now, using monotonicity properties of the energy, we can in fact control energy at time zero:

\begin{proposition}[Energy bound at time zero]\label{tlar}  Let $T$ be large.  Then
$$ \tilde E^{[T \log T, 2T \log T]_{\Z^\ast}}(0) = o_{T \to \infty}( T \log^3 T ).$$
\end{proposition}

Proposition \ref{tlar} will be proven\footnote{We thank Ofer Zeitouni for pointing out an error in a previous version of this argument.} as follows.  First we locate a good initial interval $I^0 = [I_-^0, I_+^0]$:

\begin{proposition}[Locating a good interval]\label{good}  Let $T$ be large.  Then there exists an interval $I^0 = [I_-^0, I_+^0]_{\Z^\ast}$ containing $[0.9 T \log T, 2.1 T \log T]_{\Z^\ast}$ and contained in $[0.8 T \log T, 2.2 T \log T]_{\Z^\ast}$ such that
$$  \int_{\Lambda/4}^0 \sum_\pm \sum_{1 \leq 2^n \leq 0.1 T \log T} 2^{-n} \tilde E^{[I_\pm^0 - 2^n, I_\pm^0 + 2^n]_{\Z^\ast}}(t)\ dt = \tilde O(1)$$
where $\pm$ ranges over both choices of sign $+,-$ and $n$ ranges over natural numbers with $1 \leq 2^n \leq 10 T \log T$.
\end{proposition}

Recall that $\tilde O(1)$ is any quantity which is $O(\log^{O(1)} T)$.

\begin{proof}  By the pigeonhole principle, it suffices to show that
$$ 
\int_{0.8 T \log T}^{0.9 T \log T} \int_{2.1 T \log T}^{2.2 T \log T} \int_{\Lambda/4}^0  \sum_\pm \sum_{1 \leq 2^n \leq 0.1 T \log T} 2^{-n} \tilde E^{[I_\pm^0 - 2^n, I_\pm^0 + 2^n]_{\Z^\ast}}(t)\ dt dI_+^0 dI_-^0 \lessapprox T^2.$$
By the triangle inequality (and the definition of $\tilde O$) it suffices to show that
$$ 
\int_{0.8 T \log T}^{0.9 T \log T} \int_{2.1 T \log T}^{2.2 T \log T} \int_{\Lambda/4}^0 \tilde E^{[I_\pm^0 - 2^n, I_\pm^0 + 2^n]_{\Z^\ast}}(t)\ dt dI_+^0 dI_-^0 \lessapprox 2^n T^2$$
for either choice of sign $\pm$ and any $1 \leq 2^n \leq 0.1 T \log T$.  But from the Fubini--Tonelli theorem and the definition of modified energies $\tilde E^I$ we see that
$$ \int_{0.8 T \log T}^{0.9 T \log T} \int_{2.1 T \log T}^{2.2 T \log T} 
\tilde E^{[I_\pm^0 - 2^n, I_\pm^0 + 2^n]_{\Z^\ast}}(t)\ dI_+^0 dI_-^0 \lessapprox 2^n T \tilde E^{[0.5 T \log T, 3 T \log T]_{\Z^\ast}}(t)$$
and the claim now follows from Theorem \ref{prelim}.
\end{proof}

Once the interval $I^0$ is located, the main step is to iterating the following claim:

\begin{proposition}[Energy propagation inequality]\label{tlar-iter}  Let $T$ be large, let $I = [I_-,I_+]_{\Z^\ast}$ be an interval containing $[T \log T, 2T \log T]_{\Z^\ast}$ and contained in $[0.5 T \log T, 3T\log T]_{\Z^\ast}$, let $I^0$ be as in Lemma \ref{good}, and let $\Lambda/4 \leq t_1 \leq t_2 \leq 0$ be such that $t_2 \leq t_1 + \frac{1}{100 \log^2 T}$.  Then
$$ \tilde E^{I'}(t_2) \leq \tilde E^{I}(t_1) + \tilde O(1 + |I_- - I_-^0| + |I_+ - I_+^0|),$$
where $I' \coloneqq {[I_- + \log^3 T, I_+ - \log^3 T]_{\Z^\ast}}$ is a slightly shrunken version of $I$.
\end{proposition}

Let us assume Proposition \ref{tlar-iter} for the moment and finish the proof of Proposition \ref{tlar}.  From Theorem \ref{prelim} we have
$$ \int_{\Lambda/4}^0 \tilde E^{[0.5T \log T, 3T \log T]_{\Z^\ast}}(t) \, dt = o_{T \to \infty}( T \log^3_+ T )$$
and so by the pigeonhole principle, we may find $\Lambda/4 \leq t_0 \leq 0$ such that
$$ \tilde E^{[0.5T \log T, 3T \log T]}(t_0) = o_{T \to \infty}( T \log^3_+ T ).$$
In particular
$$ \tilde E^{I^0}(t_0) = o_{T \to \infty}( T \log^3_+ T ).$$
Applying Proposition \ref{tlar-iter} $O(\log^2 T)$ times to get from $t_0$ to $0$ (starting from the interval $I^0$ and shrinking it by at most $O(\log^5 T)$ during the entire process), we conclude that
$$ \tilde E^{[I_-, I_+]}(0) \leq o_{T \to \infty}( T \log^3_+ T )$$
\sloppy for some interval $[I_-, I_+]_{\Z^\ast}$ containing $[T \log T, 2T \log T]_{\Z^\ast}$ and contained in $[0.5 T \log T, 3T\log T]_{\Z^\ast}$. 
Since $\tilde E^I(0)$ is monotone in $I$, Proposition \ref{tlar} follows.

It remains to establish Proposition \ref{tlar-iter}.  We use an argument due to Bourgain \cite[\S 4]{bourgain} that combines local conservation laws (or, in this case, local monotonicity formulae) with the pigeonhole principle.

The first step is to locate a good subset of particles indexed by an interval close to $[I_-, I_+]$ which does not gain too much energy due to interactions its environment, due to separation between these particles and the environment.  From \eqref{xji} and the pigeonhole principle, one can find natural numbers 
$$ I_- \leq j_--1 < j_- \leq I_- + \log^3 T \leq I_+ - \log^3 T \leq j_+ < j_++1 \leq I_+$$
such that
\begin{equation}\label{xjmin}
 x_{j_-}(t_2) - x_{j_--1}(t_2) \geq \frac{1}{\log T}
\end{equation}
(say) and similarly
$$ x_{j_++1}(t_2) - x_{j_+}(t_2) \geq \frac{1}{\log T}.$$

From Lemma \ref{ident}(iv) applied to $K = \{ j_--1, j_-\}$ we have
$$ \partial_t ( x_{j_-}(t) - x_{j_--1}(t) )^2 \leq 8 $$
for all $t_1 \leq t \leq t_2$.  Since $t_2 - t_1 \leq \frac{1}{100 \log^2 T}$, we conclude from the fundamental theorem of calculus and \eqref{xjmin} that
\begin{equation}\label{xjmin2}
 x_{j_-}(t) - x_{j_--1}(t) \gg \frac{1}{\log T}
\end{equation}
for all $t_1 \leq t \leq t_2$.  Similarly
\begin{equation}\label{xjmax2}
 x_{j_++1}(t) - x_{j_+}(t) \gg \frac{1}{\log T}.
\end{equation}
The basic point is that because the particles $x_{j_-},\dots,x_{j_+}$ never get too close to the remaining particles $x_j, j < j_-$ and $x_j, j > j_+$ in the system, the total energy of former set of particles will remain approximately conserved over short periods of time thanks to Lemma \ref{ident}.  More precisely, let $K$ now denote the discrete interval $K \coloneqq [j_-, j_+]_{\Z^\ast}$, and define the un-normalized energy
$$ E^K(t) \coloneqq \sum_{k,k' \in K:\, k \neq k'} E_{kk'}(t).$$
From Lemma \ref{ident} we have
$$ \partial_t E^K(t) \leq \sum_{\substack{j \not \in K \\ k,k' \in K:\,  k \neq k'}} \frac{4}{(x_k-x_{k'})^2 (x_k-x_j)(x_{k'}-x_j)}.$$
for $t_1 \leq t \leq t_2$.  But from \eqref{xjmin2}, \eqref{xjmax2}, \eqref{xji} we have
\begin{align*}
 \sum_{\substack{j \not \in K \\ k,k' \in K:\, |x_k-x_{k'}| \geq 1}}& \frac{4}{(x_k-x_{k'})^2 (x_k-x_j)(x_{k'}-x_j)} \\
&\lessapprox 
 \sum_{k,k' \in K:\, |x_k-x_{k'}| \geq 1} \frac{1}{(x_k-x_{k'})^2} \sum_\pm \left( \frac{1}{1+|k - j_\pm|} + \frac{1}{1+|k' - j_\pm|} \right)\\
&\lessapprox 1
\end{align*}
and
\begin{align*}
 \sum_{\substack{j \not \in K \\ k,k' \in K:\, |x_k-x_{k'}| < 1}} &\frac{4}{(x_k-x_{k'})^2 (x_k-x_j)(x_{k'}-x_j)} \\
&\lessapprox
 \sum_{k,k' \in K:\, |x_k-x_{k'}| < 1} \frac{1}{(x_k-x_{k'})^2} \sum_\pm \frac{1}{1+|k - j_\pm|}\\
&\lessapprox (1 + |I_- - I_-^0| + |I_+ - I_+^0|)\\
&\quad\quad \times \sum_{k,k' \in K:\, |x_k-x_{k'}| < 1} \frac{1}{(x_k-x_{k'})^2} \sum_\pm \frac{1}{1+|k - I^0_\pm|} \\
&\lessapprox (1 + |I_- - I_-^0| + |I_+ - I_+^0|)\\
&\quad\quad \times
\left(\sum_\pm \sum_{1 \leq 2^n \leq 0.1 T \log T} 2^{-n} \tilde E^{[I_\pm - 2^n, I_\pm + 2^n]}(t)  
+ T^{-1} \tilde E^{[0.5 T \log T, 3 T \log T]_{\Z^\ast}}(t)\right).
\end{align*}
Inserting these bounds and integrating using the fundamental theorem of calculus, Proposition \ref{good}, and Theorem \ref{prelim},
, we conclude that
$$ E^K(t_2) \leq E^K(t_1) + \tilde O(1 + |I_- - I_-^0| + |I_+ - I_+^0|)$$
which by monotonicity of $E^K$ in $K$ implies that
$$ E^{I'}(t_2) \leq E^I(t_1) + \tilde O(1 + |I_- - I_-^0| + |I_+ - I_+^0|).$$
Applying Lemma \ref{eb}, we conclude that
$$ \tilde E^{I'}(t_2) \leq \tilde E^I(t_1) + \tilde O(1 + |I_- - I_-^0| + |I_+ - I_+^0|) + 2 \sum_{\substack{j \in I \backslash I' \\ k \in I:\, j \neq k}} \frac{1}{(\xi_j-\xi_k)^2}$$
(the factor of two coming because if $j,k \in I$ are not both in $I'$ then at least one of the cases $j \in I \backslash I', k \in I$ and $k \in I \backslash I', j \in I$ occurs). But from \eqref{xidif} one has
$$ \sum_{\substack{j \in I \backslash I' \\ k \in I:\, j \neq k}} \frac{1}{(\xi_j-\xi_k)^2} \lessapprox 1$$
and Proposition \ref{tlar-iter} follows.

\section{Contradicting pair correlation}

It remains to see that Proposition \ref{tlar} is in contradiction with results that are known to be the case for the points $x_j(0)$. Note in particular that
$$
\sum_{T \log T \leq j, j+1 \leq 2T \log T} \frac{1}{|\xi_{j+1}-\xi_j|^2} V\left(\frac{x_{j+1}(0) - x_j(0)}{\xi_{j+1}-\xi_j}\right) \leq \tilde{E}^{[T\log T, 2T \log T]}(0).
$$
In this range using \eqref{xia} and \eqref{add-2} we have $\xi_{j+1}-\xi_j \sim 4\pi/\log_+T$, and so Proposition \ref{tlar} implies that
$$
\log^2 T \sum_{T \log T \leq j, j+1 \leq 2T \log T} V\left(\frac{x_{j+1}(0) - x_j(0)}{\xi_{j+1}-\xi_j}\right) = o_{T\rightarrow\infty}(T \log^3 T)
$$
By Markov's inequality (see \cite[Ch. 1]{TaVu}), this implies that
$$ V\left(\frac{x_{j+1}(0) - x_j(0)}{\xi_{j+1}-\xi_j}\right) = o_{T \to \infty}(1)$$
for a fraction $1-o_{T \to \infty}(1)$ of $j \in [T\log T, 2 T\log T]$.  But using the properties \eqref{vlog} of the function $V$, this implies that
$$
\frac{x_{j+1}(0) - x_j(0)}{\xi_{j+1}-\xi_j} = 1 + o_{T\to \infty}(1)
$$
or
\begin{equation}\label{picketfence}
x_{j+1}(0) - x_j(0) = \frac{4\pi + o_{T\to \infty}(1)}{\log T},
\end{equation}
for a fraction $1-o_{T \to \infty}(1)$ of $j \in [T\log T, 2 T\log T]$.

In particular since the points $x_j(0)$ are twice the imaginary ordinates of nontrivial zeroes of the Riemann zeta function, this implies that the gaps between the zeroes of the zeta function are rarely much larger or smaller than the mean spacing. But this contradicts perhaps most strikingly results of Montgomery \cite{montgomery} who determined on the Riemann Hypothesis the pair correlation measure for the zeroes, measured against a class of band-limited functions. As noted by Montgomery his result implies that a positive proportion of zeroes have a spacing between them strictly smaller than mean spacing. The proof of this claim is not written down in \cite{montgomery}, but Conrey, et. al. prove as their main result of \cite{cgggh} (using different ideas) that for any $\lambda > .77$ there exists a constant $c(\lambda) > 0$ such that at least a proportion $c(\lambda)$ of $j \leq T \log T$ satisfy 
$$
x_{j+1}(0) - x_j(0) \leq \lambda \frac{4\pi}{\log T}.
$$
This contradicts \eqref{picketfence} and therefore the assumption that $\Lambda < 0$.


\enddocument